\documentclass[a4paper,10pt,twoside]{article}
\pdfoutput=1

\usepackage[activate={true,nocompatibility},final,tracking=true,kerning=true,spacing=true,factor=1100,stretch=0,shrink=0]{microtype}
 \SetTracking{encoding={*}, shape=sc}{30}
 \microtypecontext{spacing=nonfrench}

\usepackage[T1]{fontenc} 
\usepackage[utf8]{inputenc}
\usepackage[british]{babel}
\usepackage{mathrsfs}
\usepackage{amsfonts}
\usepackage{amsmath} 
\usepackage{amsthm}  
\usepackage{amssymb}
\usepackage{mathtools}
\usepackage{graphicx}
\usepackage{xcolor}
\usepackage{url}
\usepackage{cite}
\usepackage{xifthen}
\usepackage{stmaryrd}
\usepackage{titling}
\usepackage{centernot}
\usepackage{indentfirst}
\usepackage[hang,flushmargin]{footmisc}
\usepackage[affil-it]{authblk}
\usepackage{hyperref}
\hypersetup{colorlinks=false, pdfborder={0 0 0}}

\usepackage[inline]{enumitem}
\setdescription{font=\normalfont\itshape}     
\usepackage{fancyhdr}
\newcommand{\from}{\colon}
\newcommand{\into}{\hookrightarrow}
\newcommand{\implica}{\rightarrow}
\newcommand{\coimplica}{\leftrightarrow}
\renewcommand{\phi}{\varphi}
\renewcommand{\epsilon}{\varepsilon}
\renewcommand{\models}{\vDash}
\newcommand{\proves}{\vdash}
\newcommand{\pfin}{{\mathscr P}_{\mathrm{fin}}}
\newcommand{\restr}{\upharpoonright}
\newcommand{\monster}{\mathfrak U}
\newcommand{\smallsubset}{\mathrel{\subset^+}}
\newcommand{\smallprec}{\mathrel{\prec^+}}
\newcommand{\satext}{\mathrel{^+\!\!\succ}}
\newcommand{\invtypes}{S^{\mathrm{inv}}}
\newcommand{\invext}{\mid}
\newcommand{\bla}[4]{{#1}_{#2}#3\ldots#3{#1}_{#4}}
\newcommand{\wort}{\mathrel{\perp^{\!\!\mathrm{w}}}}
\newcommand{\nwort}{\mathrel{{\centernot\perp}^{\!\!\mathrm{w}}}}
\newcommand{\invtilde}{\operatorname{\widetilde{Inv}}(\monster)}
\newcommand{\invtildeof}[1]{\operatorname{\widetilde{Inv}}({#1})}
\newcommand{\invbar}{\operatorname{\overline{Inv}}(\monster)}
\newcommand{\invbarof}[1]{\operatorname{\overline{Inv}}({#1})}
\newcommand{\equidom}{\mathrel{\equiv_\mathrm{D}}}
\newcommand{\nequidom}{\mathrel{{\centernot\equiv}_{\mathrm{D}}}}
\newcommand{\ndomeq}{\mathrel{{\centernot\sim}_{\mathrm{D}}}}
\newcommand{\domeq}{\mathrel{\sim_\mathrm{D}}}
\newcommand{\doms}{\mathrel{\ge_\mathrm{D}}}
\newcommand{\ndoms}{\mathrel{{\centernot\ge}_{\mathrm{D}}}}

\newcommand{\pow}[2]{#1^{(#2)}}
\newcommand{\inverse}{^{-1}}
\newcommand{\allora}{\Rightarrow}
\newcommand{\sse}{\Leftrightarrow}
\newcommand{\eq}{\mathrm{eq}}
\newcommand{\cldoms}{\mathrel{\triangleright}}

\newcommand{\cldomeq}{\mathrel{\bowtie}}
\newcommand{\opsc}[1]{\operatorname{\textsc{#1}}}
\newcommand{\actson}{\curvearrowright }
\DeclareMathOperator{\ct}{ct}

\DeclareMathOperator{\tp}{tp}
\DeclareMathOperator{\dcl}{dcl}
\DeclareMathOperator{\aut}{Aut}

\DeclarePairedDelimiter{\set}{\{}{\}}
\DeclarePairedDelimiter{\abs}{\lvert}{\rvert}
\DeclarePairedDelimiter{\class}{\llbracket}{\rrbracket}
\DeclarePairedDelimiter{\seq}{(}{)}

\theoremstyle{definition}
\newtheorem{defin}{Definition}[section]
\newtheorem{thm}[defin]{Theorem}
\newtheorem{pr}[defin]{Proposition}
\newtheorem{co}[defin]{Corollary}
\newtheorem{lemma}[defin]{Lemma}
\newtheorem{notation}[defin]{Notation}
\newtheorem{eg}[defin]{Example}
\newtheorem{rem}[defin]{Remark}
\newtheorem{fact}[defin]{Fact}
\newtheorem{ass}[defin]{Assumption}
\newtheorem{property}[defin]{Property}
\newtheorem{assdrp}[defin]{Assumption Drop}
\newtheorem{question}[defin]{Question}
\newtheorem{prob}[defin]{Problem}
\newtheorem{cntrex}[defin]{Counterexample}
\newtheorem*{claim}{Claim}
\newtheorem{alphthm}{Theorem}

\let\oldqed\qedsymbol
\newcommand{\qedclaim}{\mbox{$\underset{\textsc{claim}}{\oldqed}$}}
\makeatletter
\newenvironment{claimproof}[1][\it Proof of Claim]{
\let\qedsymbol\qedclaim
  \par
  \pushQED{\qed}%
  \normalfont \topsep6\p@\@plus6\p@\relax
  \trivlist
\item[\hskip\labelsep
  \upshape
  #1\@addpunct{.}]\ignorespaces
}{%
  \popQED\endtrivlist\@endpefalse
}
\makeatother
\let\qedsymbol\oldqed

\makeatletter
\newcommand{\subjclass}[2][2020]{%
  \let\@oldtitle\@title%
  \gdef\@title{\@oldtitle\footnotetext{\hspace*{-2em}#1 \emph{Mathematics subject classification.} #2}}%
}
\newcommand{\keywords}[1]{%
  \let\@@oldtitle\@title%
  \gdef\@title{\@@oldtitle\footnotetext{\hspace*{-2em}\emph{Keywords.} #1.}}%
}
\makeatother

\author{Rosario Mennuni%
  \thanks{email: \url{R.Mennuni@posteo.net} \textsc{orcid}: \url{https://orcid.org/0000-0003-2282-680X}}}
\affil{University of Leeds}
\title{The domination monoid in o-minimal theories}

\begin{document}
\maketitle
\begin{abstract}
We study the monoid of global invariant types modulo domination-equivalence in the context of o-minimal theories. We reduce its computation to the problem of proving that it is generated by classes of $1$-types. We show this to hold in Real Closed Fields, where generators of this monoid correspond to invariant convex subrings of the monster model. Combined with~\cite{ehm}, this allows us to compute the domination monoid in the weakly o-minimal theory of Real Closed Valued Fields.
\end{abstract}
One of the major areas of contemporary model-theoretic research concerns o-minimal structures, a class of ordered structures introduced in~\cite{ominI} to which  techniques  from stability theory can be generalised. In this work we further this generalisation program by developing, under the assumption of o-minimality, the theory of \emph{domination}, a notion originally arising in the stable context.

Fix a complete first-order theory $T$ with infinite models, a monster model $\monster$ of $T$, and consider the space $S(\monster)$ of \emph{global types}, that is, types over $\monster$, in any finite number of variables. The preorder $\doms$ of \emph{domination} on $S(\monster)$ is defined by declaring that $p(x)\doms q(y)$ iff there is a small type $r(x,y)$ consistent with $p(x)\cup q(y)$ such that $p(x)\cup r(x,y)\proves q(y)$. The induced equivalence relation $\domeq$, \emph{domination-equivalence}, is at the center of a deep classical result of stability theory: that in superstable theories, every element of $S(\monster)$ is domination-equivalent to a finite \emph{product} of regular types, where a realisation of the product $p\otimes q$ consists of a pair of forking-independent realisations of $p$ and $q$. This product is associative and, in stable theories, domination-equivalence is a congruence with respect to it, hence one may consider the quotient semigroup, the \emph{domination monoid} $\invtilde$. By the aforementioned  result, together with the properties of regular types and the theory of ``a-models'', in superstable theories $\invtilde$ is a free commutative monoid, and parameterises ``a-prime'' extensions of $\monster$ over finite tuples; that is, initial objects in a certain category of sufficiently saturated elementary extensions of $\monster$ which are finitely generated as such. Domination between types describes how a-prime extensions over their realisations embed in each other, and the product determines how they can be amalgamated independently.

The domination monoid can  be defined in a broader framework, with some caveats.  First,  the product $\otimes$ is in general only defined on the (dense) subspace $\invtypes(\monster)$ of \emph{invariant} types: the fixed points, under the natural action $\aut(\monster)\actson S(\monster)$, of the pointwise stabiliser $\aut(\monster/A)$ of some small  $A\subset \monster$. Further downsides are that the definition is slightly more involved than in the stable case (Definition~\ref{defin:product}),  and that the  characterisation via nonforking is lost; on the positive side, this approach allows us to define a semigroup  $(\invtypes(\monster), \otimes)$  in arbitrary first-order theories.  The general theory of its interaction with domination was developed in~\cite{invbartheory}, which isolated certain sufficient conditions ensuring domination-equivalence to be a congruence with respect to the product of invariant types, hence allowing us to define the domination monoid $\invtilde$ as $(\invtypes(\monster), \otimes)/\domeq$. Unfortunately, in~\cite{invbartheory} it was also shown that  $\invtilde$ is not well-defined in general: there are theories where domination-equivalence is not a congruence with respect to $\otimes$. Such theories must be unstable, and while the counterexample from~\cite{invbartheory} is supersimple, it is currently unknown whether $\invtilde$ is well-defined in every $\mathsf{NIP}$ theory. Its study in o-minimal theories is a first step towards a solution of this problem.

But what are the reasons for looking at $\invtilde$ in the first place? The original motivation from~\cite{hhm} was to prove the following Ax--Kochen--Er\v sov-type result. If $\monster$ is a model of the theory $\mathsf{ACVF}$ of algebraically closed valued fields, $k(\monster)$ its residue field, and $\Gamma(\monster)$ its value group, respectively  an algebraically closed field and a divisible ordered abelian group,  then $\invtilde\cong\invtildeof{k(\monster)}\oplus \invtildeof{\Gamma(\monster)}$. While it is easy to see that the domination monoid of any saturated algebraically closed field is isomorphic to the natural numbers with the usual sum, some further work is required to understand $\invtilde$ in the o-minimal theory $\mathsf{DOAG}$ of divisible ordered abelian groups. Again in \cite{hhm}, it was shown that if $\monster\models\mathsf{DOAG}$ then $\invtilde$ is a free commutative idempotent monoid.\footnote{To be precise, \cite{hhm} works with $\invbar$, an object of which $\invtilde$ is a quotient. The two happen to coincide in  $\mathsf{ACVF}$, in $\mathsf{DOAG}$, as we will later show, and in  algebraically closed fields, as can be proven by using the notion of ``weight'' from stability theory.} Our first main theorem is a generalisation of this result. While the general strategy of proof is inspired by the~\cite{hhm} treatment of $\mathsf{DOAG}$, our methods are more general, and allow for a uniform treatment of different o-minimal theories.
\begin{alphthm}[Theorem~\ref{thm:omincharmodgen}]\label{thm:red1tp}
  Let $T$ be an o-minimal theory and assume that every global invariant type is domination-equivalent to a product of $1$-types. Then $\invtilde$ is a well-defined, free commutative idempotent monoid.  Its generators may be identified with any maximal set of pairwise weakly orthogonal global invariant $1$-types.
\end{alphthm}

After recalling some preliminaries in Section~\ref{sec:prelim}, we will prove Theorem~\ref{thm:red1tp} in Section~\ref{sec:reduction}, together with some auxiliary results which will allow us to show that certain concrete o-minimal theories satisfy its hypothesis. We  will prove this to be the case in o-minimal theories where all types are simple (in the sense of~\cite{mayer}) in Section~\ref{sec:dloeco}, while in Section~\ref{sec:doag} we close a gap in the~\cite{hhm}  $\mathsf{DOAG}$ proof, and we see how in this theory $\invtilde$ is isomorphic to the domination monoid of the rank of the (group-theoretic) Archimedean valuation on $\monster$. Next, we consider the theory $\mathsf{RCF}$ of real closed fields in Section~\ref{sec:rcf} and achieve a similar reduction with the help of another Archimedean valuation, this time the field-theoretic one. 
\begin{alphthm}[Theorem~\ref{thm:rcf}]\label{thm:rcfintro}
In $\mathsf{RCF}$, the monoid $\invtilde$ is well-defined and isomorphic to the domination monoid of the Archimedean rank of the Archimedean value group of $\monster$. In other words, it is isomorphic to the free commutative idempotent monoid with generators the invariant convex subrings of $\monster$.
\end{alphthm}
 Theorem~\ref{thm:rcfintro} can be used to complete the picture, largely painted in~\cite{ehm}, of the domination monoid in the  (weakly o-minimal)  theory $\mathsf{RCVF}$ of real closed nontrivially valued fields with convex valuation ring.  The only missing step, which we carry out in Section~\ref{sec:arcvf}, is to show that $\invtilde$ is well-defined. Combining this with~\cite[Corollary~2.8]{ehm}, which provides an isomorphism $\invtilde\cong\invtildeof{k(\monster)}\oplus \invtildeof{\Gamma(\monster)}$, we obtain the following characterisation.
 \begin{alphthm}[Theorem~\ref{thm:rcvf}]
   In $\mathsf{RCVF}$, the monoid $\invtilde$ is well-defined and isomorphic to the free commutative idempotent monoid generated by the disjoint union of the set of invariant convex subgroups of the value group with the set of invariant convex subrings of the residue field.
 \end{alphthm}
We conclude by leaving a small list of open questions in Section~\ref{sec:questions}.

\paragraph{Acknowledgements}
This research is part of the author's PhD thesis~\cite{mythesis},  funded by a Leeds Anniversary Research Scholarship and supervised by Dugald Macpherson and Vincenzo Mantova. To them, this manuscript and I are deeply indebted: virtually every word you are going to read has been influenced by their guidance, suggestions, and feedback. In particular, thanks to Macpherson for pointing out Fact~\ref{fact:sepbas}, and to Mantova for pointing out Remarks~\ref{rem:rank} and~\ref{rem:rank_follow_up}.
This paper, and in particular Section~\ref{sec:rcf}, has also benefited from several discussions with Francesco Gallinaro, to whom goes my gratitude.

\section{Preliminaries}\label{sec:prelim}
\subsection{The fine print}\label{subsec:fineprint}
In this subsection we fix some (mostly standard) conventions, notations, and abuses thereof.

As usual, the letter $T$ denotes a consistent, complete first-order theory, in a possibly multi-sorted  language $L$,  with infinite models. 
Definable  closure is denoted by $\dcl$. The set of finite subsets of $X$ is denoted by $\pfin(X)$. The set of natural numbers is denoted by $\omega$ and always contains $0$.

Symbols such as $\monster$, $\monster_1$, etc.~denote    \emph{monster} models of $T$, in the sense specified hereafter.  The reader who is happy to assume that arbitrarily large strongly inaccessible cardinals exist, may take  $\monster$ to denote a model of strongly inaccessible size $\kappa(\monster)>\abs T$ which is $\kappa(\monster)$-saturated, and let \emph{small} mean ``of size strictly less than $\kappa(\monster)$''. 
In the absence of large cardinals, formally speaking, monster models  will come with cardinals. That is, we consider pairs $(\monster, \kappa(\monster))$ such that $\monster$ is   $\kappa(\monster)$-saturated and $\kappa(\monster)$-strongly homogeneous for a strong limit  $\kappa(\monster)\ge\beth_\omega(\abs T)$, and when we say  that $A$ is \emph{small}  we mean  $\abs A<\kappa(\monster)$.   So $\monster$ may be for instance $\lambda^+$-saturated and $\lambda^+$-strongly homogeneous for some $\lambda^+>\kappa(\monster)$, but sets of size $\lambda$ are not considered small. 
In practice, we will  not mention $\kappa(\monster)$, in order not to burden our notation excessively.

The letter $A$ usually represents a small subset of $\monster$, the letters   $M$, $N$  small elementary substructures. If $A$ is small and included in $\monster$ we denote this by $A\smallsubset \monster$,  or $A\smallprec \monster$  if additionally $A\prec \monster$. 
If a model is denoted e.g.~by $N$, and not by $\monster$ or variations,    the notation $A\smallsubset N$ means that $N$ is $\abs A^+$-saturated and $\abs A^+$-strongly homogeneous, and similarly for  $M\smallprec N$.

Parameters and variables are tacitly allowed to be finite tuples unless otherwise specified.  Concatenation of tuples is denoted by juxtaposition, and so is union of sets, e.g.~$AB=A\cup B$.
Coordinates of a tuple are indicated with subscripts, starting with $0$, so for instance $a=(a_0,\ldots, a_{\abs a-1})$,  where  $\abs a$  denotes the length of $a$. When $a$ is a tuple and $B$ is a set, we may write $a\in B$ in place of $a\in B^{\abs a}$. The notation $\abs a$ will be overloaded and, if $a$ is an element of an ordered group, $\abs a$ will also denote the absolute value of $a$. 
A \emph{sequence} is a function with domain a totally ordered set, not necessarily $\omega$. To avoid confusion when dealing with a sequence of tuples, indices are written as superscripts, as in $(a^i)_{i\in I}$.  Tuples and sequences may be sometimes treated as sets, as in $a_0^0\in a^0\in (a^i)_{i\in I}$. Lowercase Latin letters towards the end of the alphabet, e.g.~$x,y,z$, usually denote tuples of variables, while  letters such as $a,b,c$ usually denote tuples of elements of a model.  We  write e.g.~$x=a$ instead of $\bigwedge_{i<\abs x} x_i=a_i$; ``definable functions'' $f$ may be tuples $(\bla f0,n)$ thereof, and we write $y=f(x)$ for $\bigwedge_{i<\abs y} y_i=f_i(x)$. If $\abs x=1$ we write $x$ instead of $x_0$.

Formulas are denoted by lowercase Greek letters. When we say \emph{$L$-formula}, we mean without parameters;  we sometimes write ``$L(\emptyset)$-formula'' for emphasis. On the contrary,  \emph{definable}  means ``$\monster$-definable''; if we  only allow parameters from $A$, we say ``$A$-definable'', ``definable over $A$'', etc.  
A \emph{partial type}, denoted  by a letter such as $\pi$ or $\Phi$, is a filter on a Boolean algebra of definable subsets of a fixed finite product of sorts. We say that $\pi$ is a partial type \emph{in variables $x$}, and write $\pi(x)$, to mean that such product of sorts is given by the sorts of the variables in $x$.  Partial types, even if we describe them with sets of formulas, are not assigned predetermined variables, in the sense that if $\pi$ is a partial type in $x$, and $y$ is of length $\abs x$ and such that for each $i<n$ the sorts of $x_i$ and $y_i$ coincide, then we do not distinguish between $\pi(x)$ and $\pi(y)$. This identification, of course, will not be used if we write for instance $\Phi(x,y)=\pi(x)\cup \pi(y)$.  
  \emph{Type over $B$} means ``complete type over $B$ in finitely many variables''.  Types are denoted by letters like $p,q,r$. We stress that, as a special case of our convention on partial types, we do not distinguish between $p(x)$ and $p(y)$ when taken in isolation, but we do in statements such as $p(x)\otimes p(y)=p(y)\otimes p(x)$.  We sometimes write e.g.~$p_x$ in place of $p(x)$, and denote with $S_x(B)$ the space of types over $B$ in variables $x$.  If $Y$ is a sort, we write $S_{Y^n}(B)$ for the space of types over $B$ in $n$ variables, each of sort $Y$; in the single-sorted case, we write $S_n(B)$. A \emph{global type} is a complete type over $\monster$. The distinction between a partial type $\Phi$ and the set it defines is sometimes blurred. Its  set of realisations in $C$ is denoted by $\Phi(C)$. When manipulating types $p(x)$, $q(y)$, say, we assume without loss of generality that the tuples of variables $x$ and $y$ are disjoint.

Realisations of global types and supersets of $\monster$, are though of as living inside a bigger monster model, which usually goes unnamed but is sometimes denoted by $\monster_1$.  Implications, definable closure, etc.~are to be understood modulo the elementary diagram $\opsc{ed}(\monster_1)$. For example $\models \phi(a)$ means $\monster_1\models \phi(a)$, and if $c\in \monster_1$ and $p\in S_x(\monster c)$ then $(p\restr \monster)\proves p$ is a shorthand for $(p\restr \monster)\cup \opsc{ed}(\monster_1)\proves p$. Deductive closures may be implicitly taken, as in ``$\set{x=a}\in S_x(\monster)$''.

\subsection{The domination monoid}
We  recall how the \emph{domination monoid} $\invtilde$ is defined, together with some of its basic properties. See~\cite{invbartheory} or~\cite{mythesis} for a more thorough treatment.

 \begin{defin}\label{defin:invtp}
  \begin{enumerate}
  \item Let $A\subseteq B$.  A  type $p(x)\in S(B)$ is \emph{$A$-invariant}  iff for all $\phi(x,y)\in L(A)$ (equivalently, for all $\phi(x,y)\in L(\emptyset)$) and $a\equiv_A b$ in $B^{\abs y}$ we have $p(x)\proves \phi(x,a)\coimplica \phi(x,b)$. A global type $p(x)\in S(\monster)$ is \emph{invariant} iff it is $A$-invariant for some $A\smallsubset \monster$. Such an $A$ is called a \emph{base} for $p$.
\item We denote by $\invtypes_x(\monster, A)$ the space of global $A$-invariant types in variables $x$, with $A$ small, and by $\invtypes_x(\monster)$ the union of all $\invtypes_x(\monster, A)$ as $A$ ranges among small subsets of $\monster$.  Denote by  $S(B)$ the union of all spaces of types over $B$ in a finite tuple of variables; similarly for, say,  $\invtypes(\monster)$.
  \item  If $p(x)\in \invtypes(\monster, A)$ and $\phi(x,y)\in L(A)$, write 
\[
(d_p\phi(x,y))(y)\coloneqq\set{\tp_y(b/A)\mid \phi(x,b)\in p, b\in \monster}
\]
The map $\phi\mapsto d_p\phi$ is called the \emph{defining scheme} of $p$ over $A$.
  \end{enumerate}
\end{defin}

If we  say that a type $p$ is invariant, and its domain is not specified and not clear from context, it is usually a safe bet to assume that $p\in S(\monster)$. Similarly if we say that a tuple has invariant type without specifying over which set.
\begin{rem}\label{rem:invariantisinvariant}
By $\abs A^+$-strong homogeneity of $\monster$, a global $p\in S_x(\monster)$ is $A$-invariant if and only if it is a fixed point of the pointwise stabiliser  $\aut(\monster/A)$ of $A$  under the usual action of $\aut(\monster)$ on $S_x(\monster)$, defined by \[f\cdot p\coloneqq\set{\phi(x, f(d))\mid \phi(x, y)\in L(\emptyset), \phi(x,d)\in p}\] If $a\in \monster_1\succ \monster$ and $a\models p$, then for every $f_1\in \aut(\monster_1/A)$ extending $f$  we have $\tp(f_1(a)/\monster)=f\cdot p$. 
\end{rem}

\begin{fact}
Let $A\smallsubset  \monster\subseteq B$ and  $p\in \invtypes_x(\monster, A)$. There is a unique $p\invext B$ extending $p$ to an $A$-invariant type over $B$,  given by requiring, for each  $\phi(x,y)\in L(A)$  and $b\in B^{\abs y}$, 
 \[
 \phi(x,b)\in p\invext B\iff \tp(b/A)\in (d_p\phi(x,y))(y)
\]  
\end{fact}
If $B$ contains both $\monster$ and $C$, we  use the notation $p\invext C$ to denote $(p\invext B)\restr C$.

It is an easy exercise to show that $p\invext B$ does not depend on the choice of a base $A$ of invariance for $p$. Similarly, given any $q(y)\in S(\monster)$ and $\phi(x,y)\in L(\monster)$, it is easy to see that for all $b$, $b'$ both realising $q$ we have $\phi(x,b)\in p\invext \monster b\iff \phi(x,b')\in p\invext \monster b'$. It follows that the notion below is well-defined.
\begin{defin}\label{defin:product}
Let   $p\in \invtypes_x(\monster,A)$ and  $q\in S_y(\monster)$. The \emph{tensor product} $p\otimes q$ is defined as follows.  Fix $b\models q$; for each $\phi(x,y)\in L(\monster)$, define\footnote{Some authors denote by $q(y)\otimes p(x)$ what we denote by $p(x)\otimes q(y)$.}
 \[
 \phi(x,y)\in p(x)\otimes q(y)\iff \phi(x,b)\in p\invext \monster b
\]
We  define inductively $\pow p1\coloneqq p(x^0)$ and $\pow{p}{n+1}\coloneqq p(x^n)\otimes\pow pn(x^{n-1},\ldots, x^0)$.
\end{defin}

\begin{fact}\label{fact:productbasics}
If $p_0,p_1,p_2\in \invtypes(\monster, A)$, then $p_0\otimes p_1\in \invtypes(\monster, A)$, and moreover $(p_0(x)\otimes p_1(y))\otimes p_2(z)=p_0(x)\otimes (p_1(y)\otimes p_2(z))$.
\end{fact}
We are thus in the presence of a semigroup $(\invtypes(\monster), \otimes)$. This semigroup is a monoid, with neutral element the unique global $0$-type, namely the elementary diagram $\opsc{ed}(\monster)$ of $\monster$. We are interested in its quotients by two equivalence relations. Before introducing them, let us clarify some standard abuse of notation. Recall the conventions on types and variables stipulated in Subsection~\ref{subsec:fineprint}.
\begin{defin}\label{defin:commute}
  We say that two global invariant types $p,q$ \emph{commute}, and write $p\otimes q=q\otimes p$, iff $p(x)\otimes q(y)=q(y)\otimes p(x)$.
\end{defin}
  Saying that $p$ and $q$ commute, in the sense above, is \emph{not} the same as saying that they commute as elements of the semigroup $(\invtypes(\monster), \otimes)$. For example, every element of every semigroup commutes with itself in the standard sense. On the other hand, in the theory $\mathsf{DLO}$ of dense linear orders without endpoints, let for example $p$ be the ($\emptyset$-invariant) $1$-type $\tp(+\infty/\monster)$ of an element larger than $\monster$. Then $p(x)\otimes p(y)\proves x>y$ and $p(y)\otimes p(x)\proves x<y$, therefore $p$ does not commute with itself in the sense of Definition~\ref{defin:commute}. Even worse, in an arbitrary theory, let $p$, $q$ be any two distinct realised global $1$-types. Since $p$ and $q$ are realised, they are easily seen to commute in the sense of Definition~\ref{defin:commute}. Nevertheless, $p\otimes q$ and $q\otimes p$ are distinct as elements of the semigroup $(\invtypes(\monster), \otimes)$, since for example they do not agree on their respective first coordinates.

  In what follows, the notation $p\otimes q=q\otimes p$ will \emph{always} mean that $p$ and $q$ commute as in Definition~\ref{defin:commute}. Similarly, if for example $\sigma$ is a permutation of $\set{0,\ldots, n}$ and we write $p_0\otimes\ldots \otimes p_n=p_{\sigma(0)}\otimes \ldots\otimes p_{\sigma(n)}$, this has to be read as $p_0(x^0)\otimes\ldots \otimes p_n(x^n)=p_{\sigma(0)}(x^{\sigma(0)})\otimes\ldots\otimes p_{\sigma(n)}(x^{\sigma(n)})$.

\begin{defin} \label{defin:domination}
Let $p\in S_x(\monster)$ and $q\in S_y(\monster)$. We say that $p$ \emph{dominates} $q$,  and write $p\doms q$, iff there are some small $A$ and some $r\in S_{xy}(A)$ such that
\begin{itemize}
\item $r\in S_{pq}(A)\coloneqq\set{r\in S_{xy}(A)\mid r\supseteq (p\restr A)\cup (q\restr A)}$, and
\item $p(x)\cup r(x,y)\proves q(y)$.
\end{itemize}
In this case, we say that $r$ is a \emph{witness} to, or \emph{witnesses} $p\doms q$. We say that $p$ and $q$ are \emph{domination-equivalent}, and write $p\domeq q$, iff $p\doms q$ and $q\doms p$, possibly witnessed by different small types $r$. We say that $p$ and $q$ are \emph{equidominant}, denoted by $p\equidom q$, iff there are some small $A$ and some $r\in S_{pq}(A)$ witnessing $p\doms q$ and $q\doms p$ simultaneously.
  \end{defin}

  \begin{eg}
     In $\mathsf{DLO}$, if $p(x)\coloneqq\tp(+\infty/\monster)$, then  $p(x)\equidom p(y)\otimes p(z)$,  easily seen to be witnessed by the unique  $r\in S_{p,\pow p2}(\emptyset)$ containing the formula $x=z$. 
  \end{eg}
\begin{eg}\label{eg:deqpushf}
In an arbitrary theory, let $f$ be a definable function with domain $\phi(x)$ and codomain $\psi(y)$, or rather a tuple of definable functions if $\abs y>1$. If $p(x)\in S(\monster)$ is such that $p(x)\proves \phi(x)$, the \emph{pushforward} $(f_*p)(y)$ is the global type $\set{\theta(y)\in \monster\mid p\proves\theta(f(x))}$. For all such $p$ and $f$, we have $p\doms f_*p$, witnessed by any small type containing $y=f(x)$. If $f$ is a bijection, then $p\equidom f_*p$.
\end{eg}

In Definition~\ref{defin:domination} we are not requiring $p\cup r$ to be a complete global type in variables $xy$; that this may not be the case for any $r$, even if $p\domeq q$, can be see in~\cite[Example~3.3]{treespaper}. See~\cite[Example~1.11.3]{invbartheory} for a theory where $\domeq$ differs from $\equidom$.

It can be shown that $\doms$ is a preorder, hence $\domeq$ is an equivalence relation.
Let $\invtilde$ be the quotient of $\invtypes(\monster)$ by $\domeq$. The partial order induced by $\doms$ on $\invtilde$ will, with abuse of notation, still be denoted by $\doms$, and we call $(\invtilde, \doms)$ the \emph{domination poset}. 
This poset has a minimum, the (unique) class of \emph{realised types}, i.e.~global types realised in $\monster$, denoted by $\class 0$.

If $T$ is  such that  $(\invtypes(\monster), \otimes, \doms)$ is a preordered semigroup, we say that $\otimes$ \emph{respects} $\doms$. In particular, then domination-equivalence is a congruence with respect to the tensor product, which induces a well-defined operation on $\invtilde$, still denoted by $\otimes$ and easily seen to have neutral element $\class 0$. We call  the  structure $(\invtilde, \otimes, \class 0, \doms)$ the  \emph{domination monoid}. We usually denote it simply by $\invtilde$, and say that \emph{$\invtilde$ is well-defined} to mean that $\otimes$ respects $\doms$; this should cause no confusion since  $\invtilde$ is \emph{always} well-defined as a poset.

Similarly, $\equidom$ is an equivalence relation, and we define the \emph{equidominance quotient} $\invbar$ to be $\invtypes(\monster)/\equidom$.  Realised types form a unique class under equidominance too, which we still denote by $\class 0$, and if $p$ is arbitrary and $q$ is realised it is easy to see that $p\otimes q=q\otimes p\equidom p$.
If $\equidom$ is a congruence with respect to $\otimes$,  we call $(\invbar, \otimes, \class 0)$ the \emph{equidominance monoid}, and say that $\otimes$ \emph{respects} $\equidom$, or that \emph{$\invbar$ is well defined}.

Unfortunately, neither $\invtilde$  nor $\invbar$ need be well-defined in general. See~\cite{invbartheory} for a counterexample, together with some domination-equivalence invariants, a number of sufficient conditions ensuring compatibility of $\otimes$ with $\doms$ and $\equidom$, and a proof of the following facts.
\begin{fact}[\!\!{\cite[Lemma~1.8]{invbartheory}}]\label{fact:invariancepreserved}
  If $p\in \invtypes_x(\monster, A)$ and $r\in S_{xy}(B)$ are such that $p\cup r$ is consistent and  $p\cup r\proves q\in S_y(\monster)$, then $q$ is invariant over $AB$. In particular, if $p\doms q$ and $p$ is invariant, then so is $q$.
\end{fact}
\begin{fact}[\!\!{\cite[Lemma~1.14]{invbartheory}}]\label{fact:atleastontheleft}
Let $p_0\in\invtypes_x(\monster)$,  suppose $p_0(x)\cup r(x,y)\proves p_1(y)$, and let $s\coloneqq r(x,y)\cup \set{z=w}$.
  Then $(p_0(x)\otimes q(z))\cup s\proves p_1(y)\otimes q(w)$.
  In particular if $p_0\doms p_1$ then $p_0\otimes q\doms p_1\otimes q$, and the same holds replacing $\doms$ with $\equidom$.
\end{fact}

\begin{co}\label{co:commthenwd}
Suppose that $q_0\doms q_1$ and, for $i<2$, we have    $p\otimes q_i\domeq  q_i\otimes p$. Then $p\otimes q_0\doms p\otimes q_1$. The same holds replacing both $\doms$ and $\domeq$ with $\equidom$. In particular, if for all $p,q\in \invtypes(\monster)$ we have $p\otimes q\domeq q\otimes p$ [resp.~$p\otimes q\equidom q\otimes p$], then  $\otimes$ respects $\doms$ [resp.~$\equidom$].
\end{co}
\begin{fact}[\!\!{\cite[Corollary~1.24]{invbartheory}}]\label{fact:pushforward}
  Suppose that $q_1$ is the pushforward  $f_*q_0$ of $q_0$, for some definable function   $f$. Then, for all $p\in \invtypes(\monster)$, we have $p\otimes q_0\doms p\otimes q_1$. Assume moreover that $f$ is a bijection, so $q_1\equidom q_0$. Then, for all $p\in \invtypes(\monster)$, we have $p\otimes q_0\equidom p\otimes q_1$.
\end{fact}

Throughout the paper, an important role will be played by the binary relation of \emph{weak orthogonality}, defined below.
\begin{defin}
  Two types $p(x),q(y)\in S(B)$ are \emph{weakly orthogonal}, denoted by $p\wort q$, iff $p(x)\cup q(y)$ implies a complete type in $S_{xy}(B)$.
\end{defin}

  \begin{rem}\label{rem:wortobviousfacts}
If $p,q\in S(\monster)$ and $p$ is invariant, then  $p\wort q$ is equivalent to $p\cup q\proves p\otimes q$, or in other words to the fact that for any $c\models q$  we have  $p\proves p\invext \monster c$. Moreover, if $p$ and $q$ are both invariant, since $p(x)\otimes q(y)$ and $q(y)\otimes p(x)$ are both completions of $p(x)\cup q(y)$, then $p\wort q$ implies that $p$ and $q$ commute.
\end{rem}

\begin{fact}[\!\!{\cite[Proposition~3.13]{invbartheory}}]\label{fact:wortpreserved}
  Suppose that $p_0, p_1\in \invtypes(\monster)$ and $q\in S(\monster)$ are such that $p_0\doms p_1$ and $p_0\wort q$. Then $p_1\wort q$. In particular, if $p\doms q$ and $p\wort q$, then $q$ is realised.
\end{fact}
As a consequence of the fact above, we may endow both $\invtilde$ and $\invbar$ with the relation induced by $\wort$, which we denote by the same symbol.

\subsection{Some o-minimal facts}
For most of this  subsection $T$ will denote an arbitrary o-minimal theory.  We will assume throughout the paper that the reader is acquainted with dcl-independence and its properties, as well as with cornerstone results such as the Monotonicity Theorem. A standard reference is the book~\cite{ttaos}.

Some of the technical facts we will use are standard, follow from lemmas scattered across the literature, or are variations thereof. In order to be as self-contained as possible, we recall them here and include proofs. We begin by defining, for technical convenience, a slightly nonstandard notion of \emph{cut}.

\begin{defin}   A  \emph{cut} in a linearly ordered set $A$ is a pair of subsets $(L, R)$ of $A$ such that $A=L\cup R$, $L< R$, and $\abs {L\cap R}\le 1$.  
\end{defin}

\begin{defin}
Let $T$ be o-minimal.    To each $p\in S_1(\monster)$ we may associate a cut $(L_p, R_p)$ in $\monster$ by setting $L_p\coloneqq\set{d\in \monster\mid p(x)\proves x\ge d}$ and  $R_p\coloneqq\set{d\in \monster\mid p(x)\proves x\le  d}$. If the cofinality of $L_p$ is small we will say that $p$ has \emph{small cofinality on the left}, and if the coinitiality of $R_p$ is small that $p$ has \emph{small cofinality on the right}.
  \end{defin}
    By saturation of $\monster$, a global $1$-type has  small cofinality simultaneously on the left and on the right if and only if it is realised, and it has neither if and only if  it is not invariant. Moreover, $p\in S_1(\monster)$ is $M$-invariant  if and only if  $M$ contains a set cofinal in $L_p$, or a set    coinitial in $R_p$.

  \begin{rem}\label{rem:omintypes}
Assume that $T$ is o-minimal. Every   $1$-type $p(x)\in S_1(B)$ is determined by a  cut in $\dcl(B)$, since it is enough to specify to which $B$-definable sets $x$ belongs, and  by o-minimality these are unions of points of $\dcl(B)$ and intervals with extremes in $\dcl(B)\cup\set{\pm \infty}$.
\end{rem}

The following lemma and corollary are essentially~\cite[Lemma~6.1]{petpil}.
\begin{lemma}\label{lemma:increasingonp}
Let $T$ be o-minimal,  $f$  an $A$-definable function, and $p\in S_1(A)$. If $f_*p=p$, then $f$ is strictly increasing on $p$, i.e.~$p(x)\cup p(y)\cup \set{x<y}\proves f(x)< f(y)$.
\end{lemma}
\begin{proof}
 By the Monotonicity Theorem, $f$ is either strictly increasing, strictly decreasing, or constant on $p$. If $f$ is constant on $p$ then $p$ is realised in $\dcl(A)$, and  $f_*p=p$ yields  $p(x)\proves f(x)=x$.  Hence, we only need to exclude that $f$ is strictly decreasing on $p$. Assume towards a contradiction this is the case.
  
Let $a\models p$, and suppose $f(a)\le a$. By assumption $f\inverse (a)$ satisfies $p$, which proves $f(x)\le x$, and so $a=f(f\inverse(a))\le f\inverse(a)$. Since $f$ is strictly decreasing on $p$ if and only if $f\inverse$ is,  by replacing $f$ with $f\inverse$ we may assume $f(a)\ge a$.

Since $f\inverse$ is strictly decreasing on $p$,  from $f(a)\ge a$ we get $a\le f\inverse (a)$. But, similarly to what we did in the previous paragraph, $f\inverse(a)\models p(x)\proves f(x)\ge x$, so $a=f(f\inverse(a))\ge f\inverse(a)\ge a$. But then $p(x)\proves f(x)=x$,  contradicting that $f$ is decreasing on $p$.
\end{proof}
Recall that $p(C)$ denotes the set of realisations of $p$ in $C$.
\begin{co}\label{co:nosandwich}
Let $T$ be o-minimal and  suppose that $a,b\models p\in S_1(B)$. Then either $p(\dcl(B a))$ and $p(\dcl(B b))$ are cofinal and coinitial in each other, or one of them lies entirely to the left of the other.
\end{co}
\begin{proof}
If none lies entirely to the left of the other, we can find without loss of generality $a^0\le b^0\le a^1$, where $b^0\in p(\dcl(B b))$ and $a^i\in p(\dcl(B a))$.  If $p$ is realised in $\dcl(B)$ we are done, so assume  both $a^i$ are in $\dcl(Ba)\setminus\dcl(B)$.  By exchange, there is an $B$-definable $f$ such that $f(a^0)=a^1$, which is increasing by Lemma~\ref{lemma:increasingonp}, so $f(b^0)\ge f(a^0)=a^1$.  Since this argument works with arbitrarily large $a^1\in p(\dcl(B a))$, this proves cofinality of $p(\dcl(B b))$ in $p(\dcl(B a))$. For coinitiality, argue symmetrically. Since $b_0\le a_1\le f(b_0)$, the same argument yields cofinality and coinitiality of $p(\dcl(Ba))$ in $p(\dcl(Bb))$.
 \end{proof}

 \begin{lemma}\label{lemma:nomoreMinv1tps}
Let $T$ be o-minimal and $M\smallprec N\smallprec \monster$.   Suppose that $p\in \invtypes_n(\monster, M)$ and that $b$ is a tuple such that  $\tp(b/\monster)$ is $M$-invariant. If $p$ is realised in $\dcl(\monster b)$, then it is realised in $\dcl(N b)$ as well. 
\end{lemma}
\begin{proof}
Suppose that for some $M$-definable function $f(y,u)$ and $d\in \monster$ we have $f(b,d)\models p$. Let $\tilde d\in N$ be such that $\tilde d\equiv_M d$. Let $\phi(z, w)\in L(M)$ and $e\in \monster$ be such that  $\phi(z,e)\in p$.  We want to show that $f(b, \tilde d)\models \phi(z,e)$. Let $h\in \aut(\monster/M)$ be such that $h(d)=\tilde d$. By $M$-invariance, $\phi(z, h\inverse(e))\in p$. Therefore $f(b,d)\models \phi(z, h\inverse(e))$, hence $b\models \phi(f(y, d), h\inverse (e))$. By applying $h$ and using that $\tp(b/\monster)$ is $M$-invariant, it follows that $b\models \phi(f(y,\tilde d), e)$, hence $f(b, \tilde d)\models \phi(z,e)$.
\end{proof}
This can be improved for points that are actually named. We will use the following lemma tacitly to assume independence of a tuple without changing the invariance base. The notation $p(x^0,x^1)\proves x^0\in \dcl(B x^1)$ means ``there is a $B$-definable function $f$ such that $p(x^0,x^1)\proves x^0=f(x^1)$''.
\begin{lemma}\label{lemma:Minvindep}
Assume that $T$ is o-minimal and let $p(x)\in\invtypes_n(\monster, M)$. If $p(x)\proves x_0\in \dcl(\monster \bla x1,{n-1})$, then $p(x)\proves x_0\in \dcl(M \bla x1,{n-1})$.
\end{lemma}
\begin{proof}
  Let $p(x)\proves x_0=f(\bla x1,{n-1}, d)$, where $f(\bla x1,{n-1}, w)$ is an $M$-definable function. Up to changing $f$, we may assume that $d$ is $M$-independent. If $\abs d=0$ we are done. Inductively, assume that the conclusion holds for $\abs d\le k$,  let $\abs d=k+1$,  
  and let $\tilde d$ be $d$ with $d_k$ replaced by some different $\tilde d_k\equiv_{Md_{<k}} d_k$. Let $b\models p$; by $M$-invariance of $p$ we have $\tilde d_k\equiv_{M b d_{<k}} d_k$. Again by $M$-invariance, $p(x)\proves x_0=f(\bla x1,{n-1}, \tilde d)$, therefore $f(\bla b1,{n-1}, \tilde d)=b_0=f(\bla b1,{n-1}, d)$. Since $\tilde d_k\equiv_{M b d_{<k}} d_k$,  by the Monotonicity Theorem there is an $M  b_{>0} d_{<k}$-definable set, say defined by $\phi(b_{>0}, d_{<k}, w_k)$, which contains $d_k$, hence also $\tilde d_k$, and where   $w_k\mapsto f(b_{>0}, d_{<k}, w_{k})$ is constant. Therefore
  \[
    p(x)\proves (\exists w_k\; \phi(x_{>0}, d_{<k}, w_{k}))\land (\forall w_k\; \phi(x_{>0}, d_{<k}, w_{k})\implica x_0=f(x_{>0}, d_{<k}, w_{k}))
  \]
It follows that $p(x)\proves x_0\in \dcl(M \bla x1,{n-1}, \bla d0,{k-1})$, and we conclude by  applying the inductive hypothesis.
\end{proof}

We  record a standard characterisation of weak orthogonality in o-minimal theories for later reference. We will need two similar-looking statements.

\begin{lemma}\label{lemma:nwortpushf}
Let $T$ be o-minimal and  $p,q\in S_1(M)$ be nonrealised, Then $p\nwort q$ if and only if  there is an $M$-definable function $f$ such that $q=f_*p$.  Moreover, every such $f$ must be a bijection on $p$.
\end{lemma}
\begin{proof}
  If there is such an $f$, since $q$ is not realised, the formulas $y=f(x)$ and $y\ne f(x)$ witness that $p(x)\cup q(y)$ has more than one completion, hence  $p\nwort q$. If instead there is no such $f$,  fix $a\models p$ and $b\models q$. By assumption $q(M)=q(\dcl(Ma))$, so the cut of $b$ in $M$ determines the cut of $b$ in $\dcl(M  a)$, hence $p(x)\cup q(y)$ is complete by Remark~\ref{rem:omintypes}, so $p\wort q$. For the ``moreover'' part note that, since $q$ is nonrealised, $f$ must be a bijection on $p$ by the Monotonicity Theorem. 
\end{proof}

\begin{lemma}\label{lemma:omin1typeequidom}
Let $T$ be o-minimal and $M\smallprec N\smallprec \monster$.  For every pair of  nonrealised $p,q\in \invtypes_1(\monster, M)$ the following are equivalent.
  \begin{enumerate}
  \item\label{point:ominnwort} $p\nwort q$.
  \item \label{point:ominequidom} $p\equidom q$.
  \item \label{point:omindomeq} $p\domeq q$.
  \item \label{point:omindefbij} There is a $\monster$-definable function $f$ such that $q=f_*p$.
  \item \label{point:ominNdefbij} There is an $N$-definable bijection $f$ on $p$ such that $q=f_*p$.
  \end{enumerate}
\end{lemma}
\begin{proof}
  As remarked in Example~\ref{eg:deqpushf} we have  $\ref{point:ominNdefbij}\allora \ref{point:ominequidom}$, while $\ref{point:ominequidom}\allora \ref{point:omindomeq}$  is trivial. Fact~\ref{fact:wortpreserved} implies  $\ref{point:omindomeq}\allora\ref{point:ominnwort}$, and 
  Lemma~\ref{lemma:nwortpushf} implies  $\ref{point:ominnwort}\sse \ref{point:omindefbij}$. The implication  $\ref{point:ominNdefbij}\allora \ref{point:omindefbij}$ is also trivial;  to see that $\ref{point:omindefbij}\allora \ref{point:ominNdefbij}$,  use Lemma~\ref{lemma:nomoreMinv1tps} to show that  $f$ may be chosen $N$-definable, and  observe that $f$ must be a bijection on $p$ by Lemma~\ref{lemma:nwortpushf}.  
\end{proof}
\begin{co}
  In every o-minimal theory, the poset $\invtilde$ is infinite.
\end{co}
\begin{proof}
  Let $\kappa_0, \kappa_1$ be two small, infinite regular cardinals. For $i<2$, let $A_i\smallsubset \monster$ have order type $\kappa_i$, and define global $1$-types $p_i(x)\coloneqq\set{x> d\mid d\in A_i}\cup \set{x<d\mid d\in \monster, d> A_i}$. Note immediately that each $p_i$ has  cofinality $\kappa_i$ on the left and, by saturation, large cofinality on the right. By our assumptions on $\monster$, we may embed in $\monster$ infinitely many infinite regular cardinals. Hence, by Fact~\ref{fact:wortpreserved}, it is enough to show that $p_0\wort p_1$. If this is not the case, by Lemma~\ref{lemma:omin1typeequidom} there is a definable bijection $f$ on $p_0$ such that $p_1=f_*p_0$, which must be continuous and monotone on $p_0$ by the Monotonicity Theorem. Since both the cofinality of $L_{p_0}$ and the coinitiality of $R_{p_0}$ are infinite, $f$ is defined on an interval $[d_0, d_1]$ such that $p_0(x)\proves d_0<x<d_1$. We may furthermore assume that $f$ is strictly monotone and continuous on $[d_0,d_1]$, up to shortening the latter.  The image of $[d_0, d_1]$ under $f$ is an interval $[d_2, d_3]$ such that $p_1(y)\proves d_2<y<d_3$. Since both $p_i$ have large cofinality on the right, but not on the left, $f$ must be increasing on $[d_0,d_1]$ and map $B_0\coloneqq\set{d\in \monster \mid d>d_0, p_0(x)\proves x>d}$ to $B_1\coloneqq\set{d\in \monster \mid d>d_2, p_1(y)\proves y>d_2}$. However, by construction, each $B_i$ has cofinality $\kappa_i$, yielding a contradiction.
\end{proof}
We conclude our preliminaries with some facts involving \emph{distality}, for which we refer the reader to~\cite{simdnd} or~\cite[Chapter~9]{simon}.
 The definition below is not the original one, but is equivalent to it by~\cite[Lemma~2.18]{simdnd}.  Note that  right to left holds in every theory by Remark~\ref{rem:wortobviousfacts}.
\begin{defin}
  A theory is \emph{distal} iff it is $\mathsf{NIP}$ and whenever $p,q\in \invtypes(\monster)$ we have $p\otimes q=q\otimes p\iff p\wort q$.
\end{defin}

\begin{fact}[\!\!{\cite[Corollary~2.30]{simdnd}}]
   O-minimal theories are distal.\label{fact:omindist}
\end{fact}
\begin{proof}
It is well-known that o-minimal theories are $\mathsf{NIP}$. If $p(x)\nwort q(y)$,  there must be $a\models p$, $b\models q$, and a nonrealised cut in $\monster$ filled by both an element of $\dcl(\monster a)$, say $f(a)$, and an element of $\dcl(\monster b)$, say $g(b)$, where $f(x)$ and $g(y)$ are definable functions. The type corresponding to this cut must be invariant by Fact~\ref{fact:invariancepreserved}. If it has small cofinality on the right, then $p(x)\otimes q(y)\proves f(x)>g(y)$, while $q(y)\otimes p(x)\proves f(x)<g(y)$. If it has small cofinality on the left, then $p(x)\otimes q(y)\proves f(x)<g(y)$, while $q(y)\otimes p(x)\proves f(x)>g(y)$
\end{proof}
\begin{lemma}\label{lemma:distalperpotimes}
Let $T$ be distal. If $q_0\wort p$ and $q_1\wort p$, then $q_0\otimes q_1\wort p$. In particular, if $p\wort q$ and $n, m\in \omega$, then $\pow pn\wort \pow q m$.
\end{lemma}
\begin{proof}
If both $q_i$ commute with $p$  then  $q_0\otimes q_1$ commutes with $p$. The last statement follows by induction.
\end{proof}

\begin{fact}[\!\!{\cite[Corollary~4.7]{invnip}}]\label{fact:wortnipinv} Let $T$ be $\mathsf{NIP}$ and  $\set{p_i\mid i\in I}$ be a family of types $p_i\in \invtypes(\monster)$ such that if $i\ne j$ then $p_i\wort p_j$. Then the partial type $\bigcup_{i\in I} p_i(x^i)$ is complete. 
\end{fact}
Under distality, this also follows from~\cite[Proposition~3.25]{distrk}.

\section{Reducing to generation by 1-types}\label{sec:reduction}
\subsection{The Idempotency Lemma}

\begin{ass}
  Until further notice, all theories we consider are o-minimal.
\end{ass}

This subsection is dedicated to the proof of this section's main lemma, namely the Idempotency Lemma~\ref{lemma:idempotent}. As its name suggests, its principal consequence is that every $1$-type is idempotent modulo equidominance. Nevertheless, this lemma will also find some technical use in certain proofs. A precursor of this result, dealing with definable types only, is~\cite[Claim~2.4]{starnotdol}, itself using~\cite[Lemma]{tresslvaluation}.
 \begin{notation}
For sets $X,R$, let $X_{<R}\coloneqq\set{x\in X\mid \forall r\in R\; x<r}$.
\end{notation}
\begin{lemma}[Idempotency Lemma]\label{lemma:idempotent}
  Let $M\smallprec N\preceq \monster$. For all $p(x)\in \invtypes_1(\monster, M)$ and $b^0\models p$ the set $p(\dcl(N b^0))$ is cofinal and coinitial in $p(\dcl(\monster b^0))$.
\end{lemma}
\begin{proof}
Without loss of generality, $p$ is not realised. We deal with the case where $p$ has small cofinality on the right, the other case being symmetrical. The bulk of this proof consists in showing that $p(\dcl(N b^0))$ is cofinal in $p(\monster b^0)$. Let $R\subseteq M$ be coinitial in  $R_p$. 

  Assume  towards a contradiction that there are an $M$-definable function $f(t,w)$    and  a tuple  $d\in \monster$ such that $p(\dcl(Nb^0))<f(b^0, d)<R$. Note  that $p(x)\proves\text{``$f(x,d)\models p$''}$, and in particular by Lemma~\ref{lemma:increasingonp} the function $f(t, d)$ is strictly increasing on $p$. Moreover, up to changing $f(t,w)$, we may assume that $d$ is $M$-independent, hence so is $b^0d$.  For $i\ge 0$, define inductively     $b^{i+1}\coloneqq f(b^i,d)$. The core of the proof consists in justifying the claim below.
  \begin{claim}
For every $\ell\in \omega$, we have $b^{\ell+1}\models p\invext \dcl(N b^0\ldots b^{\ell})$. 
\end{claim}
    Note that, by Remark~\ref{rem:omintypes} and the definitions of invariant extension and  $\otimes$,  the Claim is equivalent to saying that $p(\dcl(N b^0\ldots b^{\ell}))<b^{\ell+1}<R$, or that $(b^0,\ldots, b^{\ell+1})\models \pow p{\ell+2}\restr N$.
  \begin{claimproof}
    We argue by induction, the case $\ell=0$ holding by assumption. Assume that the Claim holds for $\ell-1$, i.e.~$\dcl(Nb^0\ldots b^{\ell-1})_{<R}<b^{\ell}<R$.  Since $b^1$ satisfies $p$ as well, if we apply the  inductive hypothesis starting with $b^1$ instead of $b^0$, we obtain that $\dcl(Nb^1\ldots b^\ell)_{<R}<b^{\ell+1}<R$. What we need to show is that $\dcl(Nb^0\ldots b^\ell)_{<R}<b^{\ell+1}<R$. Let $h(\bla u0,\ell, v)$ be an $M$-definable function, let $n\in N$ be a without loss of generality $M$-independent tuple (up to changing $h$), and suppose that $h(b^0,\ldots, b^\ell,n)\models p$. In particular,  $h(b^0,\ldots, b^\ell,n)<R$, and we need to show that $b^{\ell+1}>h(b^0,\ldots, b^\ell,n)$.

    Let $Y$ be the set of realisations of $\tp(b^0/M b^1\ldots  b^\ell n)$ in a larger monster. 
    By the Monotonicity Theorem the function  $h(u_0,b^1,\ldots, b^\ell,n)$ is  strictly increasing, strictly decreasing, or constant in $u_0$ on $Y$. In the last case, $h(b^0,\ldots, b^\ell,n)\in \dcl(N b^1\ldots  b^\ell)_{<R}$, so $b^{\ell+1}>h(b^0,\ldots, b^\ell,n)$ holds by inductive hypothesis.

    Suppose now that $h(u_0,b^1,\ldots, b^\ell,n)$ is strictly decreasing in $u_0$ on $Y$. Let $b^{-1}\in N$ be such that $b^{-1}\models p\restr Mn$.  By associativity of $\otimes$ and  inductive hypothesis $(b^1,\dots, b^\ell)\models \pow p{\ell}\invext \dcl(Nb^0)$, hence  $(b^{-1}, b^1,\dots, b^\ell)\equiv_{Mn} (b^0,\ldots, b^\ell)$ because, using the inductive hypothesis again, both tuples have type $\pow p{\ell+1}\restr Mn$. This implies that $h(b^{-1}, b^1,\ldots, b^\ell,n)<R$, and that $b^{-1}\in Y$.  Since $p\proves \text{``}x>(p\restr Mn)(\monster)\text{''}$, we have  $b^0> b^{-1}$, and we get
\[
      h(b^0,\ldots, b^\ell,n)< h(b^{-1}, b^1,\ldots, b^\ell,n)\in \dcl(N b^1\ldots  b^\ell)_{<R}
\]
and it follows that $b^{\ell+1}>h(b^0, 
\ldots, b^\ell,n)$.

If instead $h(u_0,b^1,\ldots, b^\ell,n)$ is strictly increasing in $u_0$ on $Y$, let $\tilde d\in N$ be such that $\tilde d\equiv_{Mn}d$. Let $b^\epsilon\coloneqq f(b^0, \tilde d)$. Since $p$ is $M$-invariant, from $p\proves \text{``$f(x,d)\models p\restr Mn$''}$ we obtain  $b^\epsilon\models p\restr Mn$, and as $b^\epsilon\in \dcl(N b^0)_{<R}$ we have $b^1 > b^\epsilon$. Since $p$ is invariant, from $p\proves f(x,d)>x$ we obtain $p\proves f(x,\tilde d)>x$, hence we have $b^0<b^\epsilon < b^1$. In particular both $b^0,b^\epsilon$ satisfy  $p$. It follows that $(f(t,\tilde d))_*p=p$, and by Lemma~\ref{lemma:increasingonp} $f(t,\tilde d)$ is strictly increasing on $p$; let $g(t,\tilde d)$ be its inverse. 
As $g(t,\tilde d)$ must also be strictly increasing, we have that $b^{1-\epsilon}\coloneqq g(b^1, \tilde d)>g(b^\epsilon,\tilde d)=b^0$. Since $p$ is  $M$-invariant and proves that $g(x,\tilde d)$ is the inverse of $f(x,\tilde d)$,  it also proves that $g(x,d)$ is the inverse of $f(x,d)$.   Using invariance of $p$ one more time we obtain $(g(b^1,\tilde d), b^1)\equiv_{Mn} (g(b^1, d), b^1)$, or in other words $(b^{1-\epsilon}, b^1)\equiv_{Mn} (b^0, b^1)$. Moreover,   by inductive hypothesis $(b^2,\ldots, b^\ell)\models \pow p{\ell-1}\invext Nb^0 b^1$, and  since $b^{1-\epsilon}\in \dcl(N b^1)_{<R}$ and $\otimes$ is associative,   $(b^{1-\epsilon}, b^1,\dots, b^\ell)\models \pow p{\ell+1}\restr Mn$. Again by inductive hypothesis, $(b^0,b^1,\ldots, b^\ell)\models \pow p{\ell+1}\restr Mn$ as well, therefore $(b^{1-\epsilon}, b^1,\dots, b^\ell)\equiv_{Mn} (b^0,b^1,\ldots, b^\ell)$. This implies that $h(b^{1-\epsilon}, b^1,\ldots, b^\ell,n)<R$, and that $b^{1-\epsilon}\in Y$.  Since $b^0<b^{1-\epsilon}$ we have 
\[
  h(b^0,\ldots, b^\ell,n)< h(b^{1-\epsilon}, b^1,\ldots, b^\ell,n)\in \dcl(N b^1\ldots  b^\ell)_{<R}
\]
So $b^{\ell+1}>h(b^0,\ldots, b^\ell,n)$, and we are done.
\end{claimproof}
Let $\ell\coloneqq \abs d$. By definition, we have $b^1,\ldots, b^\ell\in \dcl(M d b^0)$. Moreover, since by the Claim  $(b^0,b^1,\ldots, b^\ell)\models \pow p{\ell+1}\restr M$, the $b^i$ form an $M$-independent tuple, hence by exchange $d\in \dcl(M b^0\ldots  b^\ell)^{\abs d}$. But then we have  $b^{\ell+1}=f(b^\ell,d)\in \dcl(M b^0\ldots  b^\ell)$, in contradiction with $b^{\ell+1}\models p\invext \dcl(N b^0\ldots  b^{\ell})$. This completes the proof that $p(\dcl( N b^0))$ is cofinal in $p(\dcl(\monster b^0))$.

Finally,  if for some other $M$-definable  $f(t,w)$ the point $b^1\coloneqq f(b^0, d)\models p$ witnesses that $p(Nb^0)$ is not coinitial in $p(\dcl(\monster b^0))$, i.e.~$b^1< p(N b^0)$, then by Corollary~\ref{co:nosandwich} $p(\dcl(N b^1))<p(\dcl (Nb^0))$. Again by Lemma~\ref{lemma:increasingonp},  $f(t,d)$ is strictly increasing on $p$, hence it has an inverse, but then   $b^0=f(t, d)\inverse(b^1)> p(\dcl(Nb^1))$,   contradicting cofinality of $p(\dcl(Nb^1))$ in $p(\dcl(\monster b^1))$.
\end{proof}
\begin{co}\label{co:idempotent}
  In every o-minimal theory, every $1$-type is idempotent modulo domination-equivalence, and even modulo equidominance.
\end{co}
\begin{proof}
  Consider $p(y^1)\otimes p(y^0)$, where $p$ is $M$-invariant and without loss of generality nonrealised, say with small cofinality on the right, so $\pow p2\proves y^1>y^0$. Let  $R\subseteq M$ be  coinitial in $R_p$ and fix $N$ such that $M\smallprec N\smallprec \monster$. By the Idempotency Lemma~\ref{lemma:idempotent}, any $r(x,y)\in S_{p,\pow p2}(N)$ extending $\set{x=y_0}\cup \text{``$\dcl_{<R}(N y^0)<y^1< R$''}$ witnesses equidominance.
\end{proof}

\subsection{The reduction}
We can now move to the final steps in proving the main theorem of this section, Theorem~\ref{thm:omincharmodgen}, which  characterises $\invtilde$ assuming the statement below.
\begin{property}\label{property:1tpgen}
Every invariant type is equidominant to a product of invariant $1$-types.
\end{property}
\begin{rem}\label{rem:couldhaveuseddomeq}
  Property~\ref{property:1tpgen} has a weaker variant, replacing $\equidom$ with $\domeq$. Even if, for the sake of readability, we mostly work with $\equidom$, the proofs that follow also show that the weaker assumption is  enough to prove weaker versions of the results, where all mentions of $\equidom$ are replaced by $\domeq$. 
\end{rem}

\begin{rem}
By Lemma~\ref{lemma:omin1typeequidom} and the fact that realised types are weakly orthogonal to every type, for a sequence $\seq{q_i\mid i\in I}$ of nonrealised invariant $1$-types the following are equivalent.
\begin{enumerate}
    \item The sequence $\seq{q_i\mid i\in I}$ is a a maximal sequence of pairwise weakly orthogonal invariant $1$-types.
\item The sequence $\seq{q_i\mid i\in I}$ is a sequence of representatives for the  $\equidom$-classes of nonrealised invariant $1$-types.
\item The sequence $\seq{q_i\mid i\in I}$ is a sequence of representatives for the  $\domeq$-classes of nonrealised invariant $1$-types.
\end{enumerate}
\end{rem}
\begin{defin}
  Fix a maximal sequence $\seq{q_i\mid i\in I}$  of pairwise weakly orthogonal invariant $1$-types.  For $p\in \invtypes(\monster)$, define
  \[
    I_p\coloneqq\set{i\in I\mid p\doms q_i}
  \]
\end{defin}

\begin{pr}\label{pr:omindomchar}
Let $T$ be o-minimal satisfying Property~\ref{property:1tpgen}, and let $p\in \invtypes(\monster)$. Then $I_p$ is finite and, if $p'\in \invtypes(\monster)$, the following hold.
  \begin{enumerate}
  \item The following are equivalent.
    \begin{enumerate*}
    \item \label{point:peqpp}$p\equidom p'$.
    \item  $p\domeq p'$. 
    \item $p$ and $p'$ dominate the  same $1$-types.
    \item\label{point:ippp}  $I_{p}=I_{p'}$.
\end{enumerate*}
  \item The following are equivalent.
    \begin{enumerate*}
    \item \label{point:pdpp} $p\doms p'$.
    \item For every $q\in \invtypes_1(\monster)$, if $p'\doms q$ then $p\doms q$.
    \item \label{point:ipdipp}$I_p\supseteq I_{p'}$.
  \end{enumerate*}
\end{enumerate}
\end{pr}
\begin{proof}
  In what follows, we will freely use  Lemma~\ref{lemma:distalperpotimes} and that two types commute if and only if they are weakly orthogonal. 
Let $p\in\invtypes(\monster)$ be nonrealised. By Property~\ref{property:1tpgen}, we can write $p\equidom p_0\otimes\ldots \otimes p_m$, where  the $p_j$ are nonrealised invariant $1$-types.  By Lemma~\ref{lemma:omin1typeequidom}, Lemma~\ref{lemma:distalperpotimes}, and Fact~\ref{fact:wortpreserved},  $p$ is orthogonal to every $\equidom$-class which is not one of the $\class{p_i}$; therefore the set $I_p$ must be finite.  Moreover, since different $q_i$ are orthogonal, they commute, hence products of the form $\bla q{i_0}\otimes{i_n}$ with pairwise distinct indices do not depend on the indexing. Suppose that $I_p=\set{i_0,\ldots, i_n}$ has size $n+1$;  we prove that $p\equidom \bla q{i_0}\otimes{i_n}$ by induction on $m$.  From this, it follows easily that $\mathrm{\eqref{point:ippp}}\allora\mathrm{\eqref{point:peqpp}}$ and that $\mathrm{\eqref{point:ipdipp}}\allora\mathrm{\eqref{point:pdpp}}$. Since, in both points of the conclusion, each property trivially implies the one on its right, this suffices.

  If $m=0$, then $p_0\equidom q_{i_0}$, because otherwise by Lemma~\ref{lemma:omin1typeequidom} $p=p_0\wort q_{i_0}$ and $p\doms q_{i_0}$, so  $q_{i_0}$ is realised by Fact~\ref{fact:wortpreserved}, which is absurd.

  If $m>0$, let us focus on $p_m$. By distality and Lemma~\ref{lemma:omin1typeequidom}, $1$-types that do not commute with $p_m$ commute with every type that commutes with $p_m$. Therefore, 
by  swapping some types in $\bla p0\otimes{m-1}$, we may assume  that, for some $k<m$, no pair of types from $p_{k+1}\ldots, p_m$  commutes, but that each of $p_{k+1}\ldots, p_m$ commutes with each $p_j$ for $j\le k$,   and by inductive hypothesis $\bla p0\otimes k\equidom q_{j_0}\otimes\ldots\otimes q_{j_\ell}$ for some suitable $\bla j0,\ell\in I$. By Fact~\ref{fact:atleastontheleft},
  \[
    p=\bla p0\otimes k\otimes\bla p{k+1}\otimes m\equidom  q_{j_0}\otimes\ldots\otimes q_{j_\ell}\otimes \bla p{k+1}\otimes m
  \]
  Note that $\bla p0\otimes k\wort\bla p{k+1}\otimes m$, hence $q_{j_0}\otimes\ldots\otimes q_{j_\ell}\wort \bla p{k+1}\otimes m$ by Fact~\ref{fact:wortpreserved}.  By  maximality of $\seq{q_i\mid i\in I}$ there is $\bar \imath\in I$ such that $q_{\bar\imath}\equidom p_m$.  Since for all $j\ge k+1$ we have $p_m\nwort p_{j}$, by  Lemma~\ref{lemma:omin1typeequidom} and Fact~\ref{fact:pushforward} we obtain $\bla p{k+1}\otimes m\equidom \pow {q_{\bar\imath}}{m-k}$, and by  Corollary~\ref{co:idempotent} $\pow{q_{\bar\imath}}{m-k}\equidom q_{\bar\imath}$. Moreover $\bla p{k+1}\otimes m$ is weakly orthogonal to, hence  commutes with, $q_{j_0}\otimes\ldots\otimes q_{j_\ell}$. Again by Fact~\ref{fact:atleastontheleft},
\begin{multline*}
    p\equidom q_{j_0}\otimes\ldots\otimes q_{j_\ell}\otimes \bla p{k+1}\otimes m
\\=\bla p{k+1}\otimes m\otimes q_{j_0}\otimes\ldots\otimes q_{j_\ell}\equidom q_{\bar\imath}\otimes q_{j_0}\otimes\ldots\otimes q_{j_\ell}
  \end{multline*}
  To conclude, we need to show that the inclusion $I_p\supseteq \set{\bar\imath,\bla j0,\ell}$ (a corollary of what we just proved) cannot be strict. If it is, as witnessed by $j$, then $p\doms q_j$ but $q_j \wort q_{\bar\imath}$ and $q_j\wort q_{j_\alpha}$ for $\alpha\le \ell$. By Lemma~\ref{lemma:distalperpotimes} then \[
    p\doms q_j\wort q_{\bar\imath}\otimes q_{j_0}\otimes\ldots\otimes q_{j_\ell}\equidom p
  \]
  hence $q_j$ is realised by Fact~\ref{fact:wortpreserved}, which is absurd.
\end{proof}

\begin{co}\label{co:ominotimesasexpected}
Let $T$ be o-minimal satisfying Property~\ref{property:1tpgen}. For all $p_0,p_1\in \invtypes(\monster)$ we have  $I_{p_0\otimes p_1}=I_{p_0}\cup I_{p_1}$.  
\end{co}

\begin{proof}
  Clearly, $I_{p_0\otimes p_1}\supseteq I_{p_0}\cup I_{p_1}$.  If the inclusion is strict, there is a nonrealised $q\in \invtypes_1(\monster)$ such that $p_0\otimes p_1\doms q$, but for $i<2$ we have $p_i\wort q$. By Lemma~\ref{lemma:distalperpotimes}, $p_0\otimes p_1 \wort q$, so $q$ is realised by Fact~\ref{fact:wortpreserved}, a contradiction.
\end{proof}
With similar arguments, one shows the corollary below.
\begin{co}\label{co:wortworks}
Let $T$ be o-minimal satisfying Property~\ref{property:1tpgen}. If $p,q\in\invtypes(\monster)$, then  $p\wort q$ if and only if $p$ and $q$ dominate no common nonrealised $1$-type. Moreover, if $q\in \invtypes_1(\monster)$, then $p\doms q\iff p\nwort q$.
\end{co}

\begin{co}\label{co:ominwd}
Let $T$ be o-minimal satisfying Property~\ref{property:1tpgen}. Then $\otimes$ respects $\doms$ and $\equidom$.
\end{co}
\begin{proof}
  Suppose $p_0\doms p_1$.  By Proposition~\ref{pr:omindomchar}, this means that $I_{p_0}\supseteq I_{p_1}$. By Fact~\ref{fact:atleastontheleft} it is enough to show that, for all invariant $q$, we have  $q\otimes p_0\doms q\otimes p_1$, i.e.~$I_{q\otimes p_0}\supseteq I_{q\otimes p_1}$. Similarly, if we start with $p_0\equidom p_1$ then  $I_{p_0}=I_{p_1}$, and we want to show that $I_{q\otimes p_0}= I_{q\otimes p_1}$. Both follow at once from Corollary~\ref{co:ominotimesasexpected}.
\end{proof}

After recalling  Remark~\ref{rem:couldhaveuseddomeq}, we can state the main result of this section. 
\begin{thm}\label{thm:omincharmodgen}
  Let $T$ be an o-minimal theory and assume that every invariant type is domination-equivalent to a product of $1$-types. Then $\invtilde$ is well-defined, and  $(\invtilde, \otimes, \doms, \wort)\cong (\pfin(X), \cup, \supseteq, D)$, where $X$ is any maximal set of pairwise weakly orthogonal invariant $1$-types and $D(x,y)$ holds iff $x\cap y=\emptyset$. Moreover,  if every invariant type is equidominant to a product of $1$-types, then $\equidom$ is the same as $\domeq$, hence $\invbar=\invtilde$.
\end{thm}
\begin{proof}
  By the previous results,  $\class{p}\mapsto \set{q_i\mid i\in I_p}$ is the required isomorphism.
\end{proof}

\subsection{Reducing further}
In the pages to come we will be concerned with the study of some specific o-minimal theories.  Given $T$, because of Theorem~\ref{thm:omincharmodgen}, we are interested in showing that in $T$ every invariant type is domination-equivalent to a product of invariant $1$-types, and in giving a nice description of a maximal family of pairwise weakly orthogonal invariant $1$-types. Before undertaking this task, we prove some results that help to show that a given $T$ satisfies Property~\ref{property:1tpgen}.
\begin{defin}\label{defin:funfam}
  Let $p\in S_x(\monster)$ and $A\smallsubset \monster$. Denote by $\mathcal F^{p,1}_A$ the set of functions   $A$-definable in $T$ with domain a definable  subset of $\monster^{\abs x}$ on which $p$ concentrates and codomain $\monster^1$.\footnote{I.e.~they are single definable functions, not tuples thereof: they output a single element.}
\end{defin}
\begin{property}\label{property:theominimalthing}
Suppose that $c$ is a  $\monster$-independent tuple and let $p=\tp_x(c/\monster)$.   Then $\pi(x)\proves p(x)$, where
  \[
    \pi(x)\coloneqq\bigcup_{f\in \mathcal F^{p,1}_\emptyset} \tp_{w_f}(f(c)/\monster)\cup \set*{w_f=f(x)\Bigm| f \in \mathcal F^{p,1}_\emptyset}
  \]
\end{property}
Note that if this assumption is satisfied, and   $c$ is not $\monster$-independent, a similar statement still holds, by working with a basis $c'$ of $c$ over $\monster$ and then  adding to $\pi(x)$ the formulas isolating $\tp(c/\monster c')$. 
\begin{lemma}\label{lemma:onebannerman}
  Let $p\in \invtypes_1(\monster, M)$,  let $M\smallprec N\smallprec \monster$, and let $b\models \pow p {n+1}$. If $p$ has small cofinality on the right [resp.~left] then $p(\dcl(N b_{n}))$ is cofinal [resp.~coinitial] in $p(\monster b)$.
\end{lemma}
\begin{proof}
The case where $p$ is realised is trivial, so assume $p$  is not. Assume furthermore that $p$  has small cofinality on the right (the other case is symmetrical) and let $R\subseteq M$ be coinitial in $R_p$.  Let $f(\bla t0,n,w)$ be an $M$-definable function such that   $p(\dcl(N b_{n}))<f(\bla b0,n, d)<R$. Let $\hat b\in \monster$ be such that $\hat b\models \pow pn\restr{ Nd}$. Since  $b_n\models p\invext \monster \bla b0,{n-1}$ we have $\hat b b_n\equiv_{Nd} b$, hence $p(\dcl(N b_{n}))< f(\hat b, b_n, d)< R$. This violates the Idempotency Lemma~\ref{lemma:idempotent}.
\end{proof}

\begin{co}\label{co:addafactor}
  Let $p\in \invtypes_1(\monster, M)$ and  $b\models \pow pn$. Suppose that $c\models p$. If $c>p(\dcl(\monster b))$ or $c<p(\dcl(\monster b))$ then $(b,c)\models \pow p{n+1}$ or $(c,b)\models \pow p{n+1}$.
\end{co}
\begin{proof}
As usual, assume that $p$ has small cofinality on the right. If $c>p(\dcl(\monster b))$, then $(c,b)\models p\otimes \pow p{n}$  by definition. If $c<p(\dcl(\monster b))$, in particular $c<p(\dcl(\monster b_{0}))$. By Corollary~\ref{co:nosandwich}, we have $p(\dcl(\monster c))< b_{0}$, hence $b_{0}\models p\invext \monster c$. Since $b_1> p(\dcl(\monster b_0))\supseteq p(\dcl(N b_0))$, it follows from Lemma~\ref{lemma:onebannerman} that $b_1>p(\dcl(\monster c b_0))$, hence $b_1\models p\invext \monster c b_0$. We conclude by induction.
\end{proof}

\begin{pr}\label{pr:omin1typesgenerate}
  Let $T$ be o-minimal.  Let $p(x)\in \invtypes(\monster, M_0)$, let $c\models p$ and assume that $c$ is $\monster$-independent. 
  The following facts hold. 
\begin{enumerate}
\item\label{point:Emaximal} There is a tuple $b\in \dcl(\monster c)$  of maximal length among those satisfying a product of nonrealised invariant $1$-types.
  \item\label{point:nicerep} Let $b$ be as above, and let $q\coloneqq\tp(b/\monster)=\bla q0\otimes n$, where the $q_i$ are invariant $1$-types. Up to replacing each $q_i$ with another type $\tilde q_i$ in definable bijection with it, we may assume that for $i,j\le n$ either $q_i\wort q_j$ or $q_i=q_j$. Moreover $\bla {\tilde q}0\otimes n\equidom \bla q0\otimes n$.
\end{enumerate}
Let $b,q$ be as above, and let $M, N, N_1$ be such that each $q_i$ is $M$-invariant and $M_0\preceq M\smallprec N\smallprec N_1\smallprec \monster$.
\begin{enumerate}[resume]
 
\item \label{point:bisthere} Up to replacing $b$ with another $\tilde b\models q$, we may assume $b\in \dcl(N c)$.
  
\item\label{point:coreass} Let $b,q$ be as in the previous  points and set $r\coloneqq\tp_{xy}(cb/N_1)$.
  Then  $p(x)\cup r(x,y)\proves q(y)$ and      $q(y)\cup r(x,y)\proves \pi_M(x)$, where
  \[
 \pi_M(x)\coloneqq\!\!\!\!\bigcup_{f\in \mathcal F^{p,1}_M}\!\!\!\! \tp_{w_f}(f(c)/\monster)\cup \set*{w_f=f(x)\Bigm| f \in \mathcal F^{p,1}_M}
  \]
  \end{enumerate}
\end{pr}
\begin{proof}
  \begin{enumerate}[wide, labelwidth=!, labelindent=0pt]
  \item The element $c_0$ satisfies a product of length $1$, hence a tuple $b\in \dcl(\monster c)$ satisfying a product of nonrealised invariant $1$-types exists. Since $\abs b$ is bounded above by $\dim(c/\monster)$, there is a maximal such $b$. 
  \item If, say, $q_0\ne q_i$ but $q_0\nwort q_i$, then by Lemma~\ref{lemma:omin1typeequidom} there is a definable bijection $f_i$ such that $(f_i)_*q_i=q_0$. By Fact~\ref{fact:atleastontheleft}  and Fact~\ref{fact:pushforward}, we may replace every such $q_i$ with $\tilde q_i\coloneqq(f_i)_*q_i$ inside $q=\bla q0\otimes n$ and obtain an equidominant product of $1$-types.  
Now repeat this process on each $\equidom$-class in $\set{\class{q_0},\ldots,\class{q_n}}$.
\item   By Fact~\ref{fact:productbasics} the type $q$ is $M$-invariant. Use Lemma~\ref{lemma:nomoreMinv1tps} to obtain $\tilde b\in \dcl(Nc)$ realising $q$. 
  
\item   That $p\cup r\proves q$ is trivial, so let $f\in \mathcal F^{p,1}_M$ and consider $f(c)$. Note that  $p\doms \tp(f(c)/\monster)$ and,  since $f$ is $M$-definable,   $\tp(f(c)/\monster)$ is $M$-invariant by Fact~\ref{fact:invariancepreserved}. 
Let $p_0\coloneqq\tp(f(c)/\monster)\in S_1(\monster, M)$. 
If $p_0\wort q_i$ for every $i\le n$, then by Lemma~\ref{lemma:distalperpotimes} $p_0\wort q$, hence $bf(c)$ is a tuple in $\dcl(\monster c)$ longer than $b$  satisfying a product of $1$-types, against maximality of $\abs{b}$. Therefore there is $i\le n$ such that $p_0\nwort q_i$. Since $p_0$ and all the $q_i$ are $M$-invariant, by  Lemma~\ref{lemma:omin1typeequidom} 
there is an $N$-definable bijection $g$ such that $g_*p_0=q_i$. 
Let $b'\subseteq b$ be the subtuple of $b$ consisting of points satisfying $q_i$.
\begin{claim}
There are $a_0, a_1\in q_i(\dcl(N_1 b'))$ such that $a_0<g(f(c))<a_1$.
\end{claim}
\begin{claimproof}
Otherwise, by Corollary~\ref{co:addafactor} applied to $N_1$ instead of $\monster$, one between $g(f(c))b'$ and $b'g(f(c))$ satisfies $\pow {q_i}{{\abs {b'}}+1}\restr N_1$. Call $\hat b$ the one that does. Since $g\circ f$ is $N$-definable, and $b'\in \dcl(N c)$ by point~\ref{point:bisthere}, the tuple $\hat b$ is the image under an $N$-definable function of $c$,  hence it has $N$-invariant type; by uniqueness of invariant extensions,  $\hat b\models \pow{q_i}{{\abs {b'}}+1}$. By point~\ref{point:nicerep}, Lemma~\ref{lemma:distalperpotimes},   and Fact~\ref{fact:wortnipinv},  $g(f(c))b$  or $b g(f(c))$ satisfies a product of nonrealised invariant $1$-types, against maximality of $\abs{b}$.
\end{claimproof}
   Write $a_j=h_j(b)$, where each $h_j(y)$ is $N_1$-definable. Then, depending on whether $g$ is strictly increasing or strictly decreasing, either the formula $g\inverse(h_0(y))<f(x)<g\inverse(h_1(y))$ or its analogue with both  inequalities reversed is in $r$. Since $q(y)$ shows that both $g\inverse(h_j(y))$  realise $p_0$, we have  $q(y)\cup r(x,y)\proves \tp(f(x)/\monster)$, and  we are done. \qedhere
 \end{enumerate}
\end{proof}

\begin{co}\label{co:assgac}
  Property~\ref{property:theominimalthing} implies  Property~\ref{property:1tpgen}. 
\end{co}
\begin{proof}
 Let $p(x)=\tp(c/\monster)$ be $M_0$-invariant.  By working on a basis $c'$ of $c$ and then adding to $r$ the formulas isolating $\tp(c/\monster c')$ (see Lemma~\ref{lemma:Minvindep}),  we may assume that $c$ is $\monster$-independent. Apply Proposition~\ref{pr:omin1typesgenerate}, and obtain a product $q(y)$ of invariant $1$-types and a small  $r\in S_{pq}(N_1)$ such that  $p(x)\cup r(x,y)\proves q(y)$ and  $q(y)\cup r(x,y)\proves \pi_M(x)$. Trivially,  $\pi_M(x)\proves \pi(x)$, and by Property~\ref{property:theominimalthing}  $\pi(x)\proves p(x)=\tp(c/\monster)$.
\end{proof}

\section{Theories with no nonsimple types}
\label{sec:dloeco}

In this section we deal with o-minimal theories with ``few'' definable functions, such as $\mathsf{DLO}$. The main definition comes from~\cite{mayer}.
\begin{defin}
A $1$-type $p\in S_1(A)$ is \emph{simple} iff whenever there are an $A$-definable function $f(\bla x0,n)$ and points  $\bla a0,n$, all realising  $p$,   such that  $f(\bla a0,n)\models p$, then there is $j\le n$ with $\bigcup_{i\le n} p(x_i)\proves f(\bla x0,n)=x_j$.
\end{defin}
\begin{rem}
A $1$-type $p(x)$ is simple if and only if for all $k\in\omega$ the type $\set{\bla x0<{k}}\cup\bigcup_{i\le k} p(x_i)$ is complete.
\end{rem}
\begin{proof}
  Left to right is immediate. Right to left, suppose that $\bla a0,n$, and $f(\bla a0,n)$ all realise $p$, but for all $i\le n$ we have $f(\bla a0,n)\ne a_i$. Suppose for example that $f(\bla a0,n)> a_n$, the other cases being analogous. Then $\set{\bla x0<{n+1}}\cup\bigcup_{i\le {n+1}} p(x_i)$ is consistent with both $f(\bla x0,n)=x_{n+1}$ and $f(\bla x0,n)\ne x_{n+1}$.
\end{proof}

Note that, if $p$ is invariant, modulo reversing the order of the variables the type above must be $\pow p{k+1}$.

\begin{fact}[\!\!{\cite[Corollary~2.6]{rastsahota}}]
Let $T$ be o-minimal.  There is a nonsimple $1$-type over $\emptyset$ if and only if there is one over some $A$, if and only if there is one over every $A$.
\end{fact}

\begin{pr}
  Suppose that every $p\in \invtypes_1(\monster)$ is simple. Then every invariant type is equidominant to a product of $1$-types.
\end{pr}
\begin{proof}
Let $\tp(a/\monster)$ be invariant. By Lemma~\ref{lemma:omin1typeequidom} and Fact~\ref{fact:wortpreserved} we may assume that for all $i,j<\abs a$ either $\tp(a_i/\monster)\wort \tp(a_j/\monster)$ or $\tp(a_i/\monster)= \tp(a_j/\monster)$. Furthermore, by Lemma~\ref{lemma:distalperpotimes} and Fact~\ref{fact:wortnipinv}, it is enough to show that if the $a_i$ all have the same type $p$, then $\tp(a/\monster)\equidom \pow p{\abs a}$; equivalently,  by the Idempotency Lemma, that $\tp_x(a/\monster)\equidom p(y)$. This is immediate from the definition of simplicity, by taking as $r(x,y)$ a small type containing $y=x_i$, where $x_i=\min \set{\bla x0,{\abs x-1}}$ if $p$ has small cofinality on the right, and $x_i=\max \set{\bla x0,{\abs x-1}}$ otherwise.
\end{proof}
  \begin{defin}
    We call a cut in   $\monster$ \emph{invariant}  iff exactly one between the cofinality of $L$ and the coinitiality of $R$ are small. Let $\operatorname{IC}(\monster)$ be the set of all such.
  \end{defin}

\begin{eg}\label{eg:dloplusmax} 
Let $T$ be $\mathsf{DLO}$, or another completion of the theory of dense linear orders, for instance that where the universe has a maximum but no minimum.  It is easy to see using quantifier elimination that every $1$-type in $T$ is simple and all pairs of distinct nonrealised invariant $1$-types are weakly orthogonal.  Therefore Theorem~\ref{thm:omincharmodgen} applies, with $X$ the set of all nonrealised invariant $1$-types, which may be identified with  $\operatorname{IC}(\monster)$.
\end{eg}

\section{Divisible ordered abelian groups}\label{sec:doag}
Let $\mathsf{DOAG}$ be the theory of divisible ordered abelian groups, in the language $L=\set{+,0,-,<}$. It is well-known (see~\cite{ttaos}) that this theory is complete, eliminates quantifiers, and is o-minimal. 

As we saw in the introduction, $\invbar$ was computed for $\mathsf{DOAG}$ in~\cite[Corollary~13.20]{hhm}, and in fact the general strategy behind Section~\ref{sec:reduction} is inspired by this result. Unfortunately, the proof in~\cite{hhm} has a gap, explained at the end of this section.  
In what follows, we still use ideas and results from~\cite{hhm}, but we  avoid altogether the part of the~\cite{hhm} proof containing the gap. There are other minor differences between the present approach and that of~\cite{hhm}. For example,  Proposition~\ref{pr:icomplsubgrp} is a consequence of~\cite[Corollary~13.11]{hhm}, but the proof given here is easier to generalise, as we do in Proposition~\ref{pr:mcomplsubring}. 

\begin{pr}
The theory $\mathsf{DOAG}$ satisfies  Property~\ref{property:theominimalthing}.
\end{pr}
\begin{proof}
  By quantifier elimination,   a type $p(x)\in S(\monster)$ is determined once all formulas of the form $f(x,d)\ge 0$  are decided, where $f(x,d)$ is a $\mathbb Q$-linear combination $f(x,d)=\sum_{i<k} \lambda_i\cdot x_i+\sum_{j<\ell} \mu_j\cdot d_j$. Rearrange $f(x,d)\ge 0$ as
  \begin{equation}
    \sum_{i<k} \lambda_i\cdot x_i\ge -\sum_{j<\ell} \mu_j\cdot d_j\label{eq:rearrangedoag}    
  \end{equation}
  Since $\sum_{i<k} \lambda_i\cdot x_i$ is an $\emptyset$-definable function, whether~\eqref{eq:rearrangedoag} or its negation holds is decided by the partial type $\pi(x)$ in Property~\ref{property:theominimalthing}, thereby proving that the latter holds.
\end{proof}

By Corollary~\ref{co:assgac} we  can therefore apply Theorem~\ref{thm:omincharmodgen} to $\mathsf{DOAG}$. We are now left with the task of identifying a nice maximal set of pairwise weakly orthogonal invariant $1$-types.

\begin{defin}\label{defin:qh}
  Let $H$ be a convex subgroup of $M\models \mathsf{DOAG}$. Let $q_H(x)$ denote the  element of $S_1(M)$ defined by $\set{x>d\mid d\in H}\cup\set{x<d\mid d\in M, d>H}$.
\end{defin}
\begin{defin}\label{defin:invconsgrp}
  If $H$ is a convex subgroup of $\monster$, we say that $H$ is \emph{invariant} iff there is a small $A$ such that $H$ is fixed setwise by $\aut(\monster/A)$.
\end{defin}
\begin{rem}\label{rem:qhinv}
If $H\le \monster$ is convex, it is easy to show, using convexity  and Remark~\ref{rem:invariantisinvariant}, that $q_H\in S_1(\monster)$ is invariant if and only if $H$ is.
\end{rem}

\begin{pr}
Whenever $H_0\subsetneq H_1$ are distinct convex subgroups of $M\models \mathsf{DOAG}$, we have  $q_{H_0}\wort q_{H_1}$.
\end{pr}
\begin{proof}
   By quantifier elimination, we need to show that knowing $a\models q_{H_0}$ and $b\models q_{H_1}$ is enough to decide, for $d_i\in \dcl(M a)$ and $e_i\in \dcl(M b)$, all the inequalities of the form  $d_0+e_0\le d_1+e_1$.  Since the $d_i$ and $e_i$ are $\mathbb Q$-linear combinations of elements in $Ma$ and $M b$ respectively, after some  algebraic manipulation, we  find $c\in M$ and $\gamma, \delta\in \mathbb Q$  such that $d_0+e_0\le d_1+e_1\iff \gamma b+\delta a\ge c$. If $\gamma=0$ then this information is in $\tp(a/M)=q_{H_0}$, and if $\delta=0$ it is in $\tp(b/M)=q_{H_1}$, hence we may assume that $\gamma>0$ and $\delta\in\set{1,-1}$. Note that, since $H_1$ is a convex subgroup, all positive rational multiples of $b$ have the same type over $M$. If $\gamma>0$ then, since  $H_0\subsetneq H_1$, 
   \[
     2\gamma b= \gamma b+\gamma b\ge \gamma b + a\ge \gamma b\ge \gamma b-a\ge \gamma b-\gamma b/2=\gamma b/2
   \]
Since  $2\gamma b$,  $\gamma b$, and $\gamma b/2$ have the same type over $M$, knowing  $a\models q_{H_0}$ and $b\models q_{H_1}$ is enough to deduce $\gamma b\pm a\equiv_M b$, hence to decide whether $\gamma b\pm a\ge c$ holds or not.  If  $\gamma<0$, argue similarly by showing that $\gamma b\pm a\equiv_M -b$.
\end{proof}

Recall the following notions from~\cite[Chapter~13]{hhm}.
\begin{defin}  Let $M\le N$ be ordered abelian groups.
\begin{enumerate}
\item If $b\in N$, let $\ct_M(b)\coloneqq\set{a\in M\mid 0\le a<b}$.
\item We call $N$  an \emph{i-extension} of $M$ iff there is no $b\in N$ such that $b>0$ and $\ct_M(b)$ is closed under addition.
\item An ordered abelian group  is \emph{i-complete} iff it has no proper i-extensions.
\end{enumerate}
\end{defin}

As pointed out in~\cite[Section~13.2]{hhm}, i-extensions of $M$ are those that ``do not add convex subgroups''. In other words, they are the extensions that do not change the \emph{rank} of the \emph{Archimedean valuation}. A valuation on an abelian group is defined similarly to a valuation on a field, by dropping the requirements on the interaction with the product; it takes values in a linear order, called its \emph{rank}. The (ordered group-theoretic) \emph{Archimedean valuation} $\mathfrak v$ on an ordered abelian group $M$ is defined by saying that $\mathfrak  v(x)>\mathfrak   v(y)$ iff for all $n\in \omega$  we have $n\cdot \abs x < \abs y$.  The rank $\mathfrak  v(M)$ of this valuation is the set of those convex subgroups of $M$ which are generated (as convex subgroups) by a single element, ordered by reverse inclusion. Note that the valuation reverses the order.

\begin{fact}[\!\!{\cite[Lemma~13.9]{hhm}}]\label{fact:icompliext}
  Every ordered abelian group $M$ has an i-complete i-extension, of size at most $\beth_2(\abs M)$.
\end{fact}

\begin{pr}\label{pr:icomplsubgrp}
   Let $p\in S_1(M)$, with $M\models\mathsf{DOAG}$  i-complete. There are an $M$-definable function $f$ and a convex subgroup $H$ of $M$ such that $f_*p=q_H$.
\end{pr}
\begin{proof}
  Let $p$ be a counterexample and let $a\models p$. By assumption, there is no $b>0$ in $\dcl(M a)$ such that $\ct_M(b)$ is closed under addition. But then $\dcl( M a)$ is an i-extension of $M$.
\end{proof}

\begin{co}\label{co:subgroups}
In $\mathsf{DOAG}$,  for every invariant $1$-type $p$ there are a definable bijection $f$ and an invariant convex subgroup $H$ of $\monster$ such that $f_*p=q_H$. 
\end{co}
\begin{proof}
Suppose that $p\in \invtypes(\monster, M)$. Up to enlarging $M$ not beyond size $\beth_2(\abs M)$,  for some $M$-definable bijection $f$ we have $f_*(p\restr M)=q_{H_0}$, where  $H_0$  is some convex subgroup of $M$. Since by Fact~\ref{fact:invariancepreserved} $f_*p$ is $M$-invariant, and extends $q_{H_0}$, it can only be one of two types; depending on whether it has small cofinality on the left or on the right, it will be the unique $M$-invariant type $q_0$ with $H_0$ cofinal in $L_{q_0}$, or the unique $M$-invariant $q_1$ with $(M\setminus H_0)_{>0}$ coinitial in $R_{q_1}$. Both of these are clearly  of the form $q_H$, where $H$ is an invariant  convex subgroup of $\monster$: the convex hull of $H_0$ in the first case, and  $\set{d\in \monster\mid \abs d<(M\setminus H_0)_{>0}}$ in the second.
\end{proof}

We sum everything up as follows. Again, most of the theorem below was essentially proven in~\cite[Corollary~13.20]{hhm}.
\begin{thm}\label{thm:doag}
In $\mathsf{DOAG}$, the domination monoid $\invtilde$ is well-defined, coincides with $\invbar$, and $(\invtilde, \otimes, \doms, \wort)\cong (\pfin(X), \cup, \supseteq, D)$, where $X$ is the set of invariant convex subgroups of $\monster$.
\end{thm}

\begin{rem}\label{rem:rank}
By saturation, the rank $\mathfrak v(\monster)$ of $\monster$ under the Archimedean valuation $\mathfrak v$ is a dense linear order with a maximum (the subgroup $\set 0$) but with no minimum. Since $\mathfrak v(\monster)$ inherits saturation and homogeneity from $\monster$, we may compute its domination monoid, which we did in Example~\ref{eg:dloplusmax}. Moreover, one easily shows that $\mathfrak v$ induces a bijection between invariant convex subgroups of $\monster$ and invariant cuts in $\mathfrak v(\monster)$, hence an isomorphism  $\invtilde\cong\invtildeof{\mathfrak v(\monster)}$. 
\end{rem}

In stable theories, $p\doms q$ can be equivalently defined as follows: there are $a\models p$ and $b\models q$ such that whenever $d$ is forking-independent from $a$ over $\monster$, then $d$ is forking-independent from $b$ over $\monster$ (see e.g.~\cite[Lemma~1.4.3.4(iii)]{gstheory}). It is natural to ask what happens if, in an o-minimal theory, we consider the domination $\cldoms^\mathrm{dcl}$ induced by dcl-independence in a similar fashion, which  is easily characterised as follows: $p\cldoms^\mathrm{dcl} q$ if and only if there are $a\models p$ and $b\models q$ such that $b\in \dcl(\monster a)$. In one respect, this notion is more well-behaved, since $p\cldoms^\mathrm{dcl} q$ clearly implies that the dimension over $\monster$ of any  realisation of $p$ is not lower than that of any realisation of $q$, while by the Idempotency Lemma this is not in general true if $p\doms q$. On the other hand, with the induced equivalence relation $\cldomeq^\mathrm{dcl}$,  it is not true that every invariant type is equivalent to a product of invariant $1$-types.
\begin{cntrex}
  In $\mathsf{DOAG}$, let $a\models p(x)\proves x>\monster$, let $\gamma\in \mathbb R\setminus \mathbb Q$,  and let $b$ be such that $b\models r(a,y)\coloneqq \set{y> \beta\cdot a\mid \beta\in \mathbb Q, \beta\le \gamma}\cup \set{y< \beta\cdot a\mid \beta\in \mathbb Q,\beta\ge \gamma}$. Note that, by this very description, $q\coloneqq\tp(ab/\monster)\equidom p$ in our usual sense. Since $\dim(ab/\monster)=2$ and $\cldomeq^\mathrm{dcl}$ respects the dimension, if $q$ is equivalent to a product of nonrealised invariant $1$-types, we may assume this product has only two factors, say $q\cldomeq^\mathrm{dcl} q_0\otimes q_1$. By using $\wort$, we can easily show that both $q_i$ must actually be interdefinable with $p$. But $\pow p2$ is not realised in $\dcl(\monster ab)$.
\end{cntrex}

I have already said that the characterisation of $\invbar$ in $\mathsf{DOAG}$ is not new. While we will not see all the details of how it is done in~\cite{hhm}, nor define all the notions involved, I would like to point out what is the gap that has been addressed above. The problem in the original proof resides, I believe, in an implicit use of symmetry of i-freeness in an unproven statement, or at least I have not been able to prove the latter without using symmetry. 

In detail, the proof of~\cite[Lemma~13.16]{hhm} uses~\cite[Lemma~13.14]{hhm}, in the proof of which it is assumed that $A'$ is i-free from $B$ over $C$. The only way I can see to show that  $a'$  exists is to use that $B$ is i-free from $A'$ over $C$. Unfortunately, i-freeness is not symmetric.

\section{Real closed fields}\label{sec:rcf}

Let $L=\set{+,0,-,\cdot,1,<}$ and denote, as usual, the (complete, o-minimal) $L$-theory of real closed fields by $\mathsf{RCF}$. To complete the study of $\invtilde$ in this theory, we need to show that classes of $1$-types generate it, and to identify a nice representative for each such class. We do this by using a dash of valuation theory. The appearance of valuations is no coincidence: it will follow from the results in this section that, if $\Gamma(\monster)$ is the value group of $\monster\models \mathsf{RCF}$ with respect to the (ordered field-theoretic) Archimedean valuation, then $\invtilde\cong\invtildeof{\Gamma(\monster)}$. 

We point out once and for all that in this section valuations are only used ``externally'', i.e.~are \emph{not} part of the language of any  structure under consideration. Beside  basic  valuation theory, we  only use some properties of \emph{Hahn fields} and of \emph{maximally complete} valued fields, for which the reader is referred to~\cite{vddvf}.

While we will not use it, the reader should be aware that  in $\mathsf{RCF}$ and its o-minimal expansions invariant types have a particularly nice characterisation:  by~\cite[Proposition~2.1]{onnipand}, a global type is $M$-invariant if and only if it only implies formulas that do not fork over $M$, and these have been characterised as those that contain a set ``halfway-definable over $M$''. See~\cite{dolich} for a proof of this fact, along with the precise definition of ``halfway-definable over $M$''.

Observe immediately that if $H$ is a subring of the ordered field $M$, then the convex hull of $H$ is a subring of $M$, and that convex subrings of an ordered field are valuation rings. These facts will be used tacitly throughout.

In this section, we will keep using Definition~\ref{defin:qh}, and Definition~\ref{defin:invconsgrp}, but only for those $H$ which are convex subrings.  Remark~\ref{rem:qhinv} still applies.

\begin{lemma}\label{lemma:notimmediate}
  Let $H$ be a convex subring of $M\models \mathsf{RCF}$, and let $c\models q_H$. Let $v$ be the valuation on $\dcl(M c)$ with valuation ring the convex hull of $H$. Then $v(c)\notin v(M)$.
\end{lemma}
\begin{proof}
  Suppose that $d\in M$ is such that $v(c)=v(d)$ and, without loss of generality, that $c$ and  $d$ are positive. By convexity of valuation balls, the fact that $v(c)<0$ by assumption, and the fact that $v(c^k)=k\cdot v(c)$, we have that if $c<d$ then $c<d<c^2$, and if $d<c$ then $c^{\frac 12}<d<c$. The set $q_H(\dcl(M c))$ is convex in $\dcl(M c)$ and closed under positive powers, as can be easily shown by using that $H$ is a subring. 
  This implies $d\models q_H$, which is absurd because $d\in M$.
\end{proof}

\begin{pr}\label{pr:convexringsswort}
Whenever $H_0\subsetneq H_1$ are distinct convex subrings of $M\models \mathsf{RCF}$, we have  $q_{H_0}\wort q_{H_1}$.
\end{pr}
\begin{proof}
  Otherwise, by Lemma~\ref{lemma:nwortpushf}, there is an  $M$-definable function $h$ such that $h_*q_{H_0}=q_{H_1}$. Let $a\models q_{H_0}$ and $b=h(a)$. On $\dcl(M  a)$, consider the valuation $v_1$ with valuation ring the convex hull  of $H_1$. By the previous lemma $v_1(b)\notin v_1(M )$ so, in order to reach a contradiction, it is enough to show $v_1(\dcl (M  a))=v_1(M )$.

  Suppose we show that, when $f(x)\in M [x]$ is a polynomial, $v_1(f(a))\in v_1(M )$. Then this holds too for  $f$ a  rational function hence, if $M (a)$ is the field generated by $M  a$, we have $v_1(M (a))=v_1(M )$. It is well-known (see~\cite[Corollary~3.20]{vddvf}) that each embedding of  valued fields $(K,v_1\restr K)\into (L,v_1)$ such that $L$ is a real closure of $K$ induces an embedding of the value group of $K$ in a divisible hull. Since $v_1(M (a))=v_1(M )$ is already divisible,  $v_1(\dcl( M  a))=v_1(M )$. 

  So, suppose that  $f(x)=\sum_{i\le n} d_ix^i$, with $n=\deg f$ and $d_i\in M $. Consider the valuation $v_0$ on $\dcl(M  a)$ with valuation ring the convex hull of $H_0$.  We show by induction on $n$ that there is $i\le n$ such that $v_0(f(a))=v_0(d_ia^i)$. For $n=0$ there is nothing to prove. Write $f(a)=d_{n+1}a^{n+1}+g(a)$, with $\deg g(a)\le n$. If $v_0(d_{n+1}a^{n+1})$ is different from $v_0(g(a))$, then $v_0(f(a))$ is the minimum of the two and we are done by inductive hypothesis. Otherwise, again by inductive hypothesis, there is $j\le n$ such that $v_0(d_{n+1}a^{n+1})=v_0(g(a))=v_0(d_j a^j)$. This implies $(n+1-j)\cdot v_0(a)=v_0(d_{j})-v_0(d_{n+1})\in v_0(M )$. Since $v_0(a)\notin v_0(M )$ by the previous lemma, we have $j=n+1$, a contradiction.

Therefore, $v_0(f(a))=v_0(d_{i}a^i)$ for some $i\le n$. Since $H_0\subsetneq H_1$, we have that $v_0(z)=v_0(w)$ implies $v_1(z)=v_1(w)$, and since $a>1$ is in the convex hull of $H_1$ we have $v_1(a)=0$, hence $v_1(f(a))=v_1(d_i a^i)=v_1(d_i)+i\cdot v_1(a)=v_1(d_i)\in v_1(M )$.
\end{proof}

\begin{defin}
Let  $M$ be an ordered field. The  valuation ring of the (ordered field-theoretic) \emph{Archimedean valuation} $v$ on $M$ is the convex hull of $\mathbb Z$.
\end{defin}
Note that this is the finest convex valuation on $M$.

\begin{fact}\label{fact:hahn}
Every real closed field $M$ admits an elementary embedding in a Hahn field $\mathbb R((t^\Gamma))\models \mathsf{RCF}$, with $\Gamma\models \mathsf{DOAG}$  the value group of $(M,v)$, where $v$ is the Archimedean valuation. Moreover,  Hahn fields $\mathbb R((t^\Gamma))$ are always maximally complete and have size at most  $2^{\abs \Gamma}$.
\end{fact}
\begin{proof}[Proof sketch]
An embedding exists by the field version of Hahn's Embedding Theorem\footnote{This result has a somehow folkloric status. See~\cite[p.187]{ehrlich} for an explanation of why it is difficult to attribute it.}. Since $M\models \mathsf{RCF}$, the group $\Gamma$ is divisible, and it is well-known that if $\Gamma\models \mathsf{DOAG}$ then $\mathbb R((t^\Gamma))\models \mathsf{RCF}$. Elementarity of the embedding follows from quantifier elimination.   See~\cite[Corollary~4.13]{vddvf} for maximal completeness, and for the size bound note that $\abs{\mathbb R((t^\Gamma))}\le (2^{\aleph_0})^{\abs \Gamma}=2^{\abs{\Gamma}}$.
\end{proof}
\begin{pr}\label{pr:mcomplsubring}
  Let $\Gamma\models \mathsf{DOAG}$ be i-complete, and let $M\coloneqq \mathbb R((t^\Gamma))$.  For every $p\in S_1(M)$ there are an $M$-definable function $f$ and a convex subring $H$ of $M$ such that $f_*p=q_H$.
\end{pr}
\begin{proof}
  Let $p$ be a counterexample, let $a\models p$, and let $N\coloneqq\dcl(M a)$. Since every point of $N$ is of the form $f(a)$, for $f$ some $M$-definable function, it is enough to find $b\in N$ such that $\tp(b/M)$ is of the form $q_H$, and take $f$ to be an $M$-definable function such that $b=f(a)$.
  
 When we look at both $M$ and $N$ equipped with the Archimedean valuation, which we call $v$ in both cases, they both have residue field $\mathbb R$. Since $M$ is maximally complete by Fact~\ref{fact:hahn}, the value group $\Gamma(N)$ of $(N,v)$ must be larger than $\Gamma$. 
Since $\Gamma$ is i-complete, by Proposition~\ref{pr:icomplsubgrp} there must be $\gamma\in \Gamma(N)$ such that $\tp(\gamma/\Gamma)=q_{\tilde H}$, where $\tilde H$ is a convex subgroup of $\Gamma$. Let $b\in N$ be such that $v(b)=-\gamma$. Since $\gamma>0$, we have $\abs b>1$, and in fact $\abs b> \mathbb R$, hence by possibly replacing $b$ with $-b$  we may assume that $b>2$. Let $H\coloneqq\set{m\in M\mid \abs m<b}$. Since $\tilde H$ is a convex subgroup of $\Gamma$ and $v(b)=-\gamma\notin v(M)$, we have that $\ct_M(b)$ is closed under products. Since $b>2$, it easily follows that $\ct_M(b)$ is also closed under sums, hence $H$ is a convex subring of $M$, and clearly $\tp(b/M)=q_H$.
\end{proof}
\begin{co}\label{co:subrings}
In $\mathsf{RCF}$,  for every invariant $1$-type $p$ there are a definable bijection $f$ and an invariant convex subring $H$ of $\monster$ such that $f_*p=q_H$. 
\end{co}
\begin{proof}
  Suppose that $p\in \invtypes(\monster, M)$, and equip $M$ with the Archimedean valuation. Note that an embedding of ordered groups  $\Gamma(M)\into \Gamma$ induces an embedding of ordered fields $\mathbb R((t^{\Gamma(M)}))\into \mathbb R((t^\Gamma))$. Using this, Fact~\ref{fact:hahn}, Fact~\ref{fact:icompliext} and the fact that $\abs{\Gamma(M)}\le \abs M$, up to enlarging $M$ not beyond size $\beth_3(\abs M)$ we may assume   that it is of the form $\mathbb R((t^\Gamma))$, with $\Gamma\models \mathsf{DOAG}$ i-complete. 

 By  Proposition~\ref{pr:mcomplsubring} there are a convex subring $H_0$ of $M$ and an   $M$-definable bijection $f$ such that  $f_*(p\restr M)=q_{H_0}$. To conclude, argue as in Corollary~\ref{co:subgroups}.
\end{proof}
 
We are left  to show that every invariant type is domination-equivalent to a product of $1$-types. The strategy of proof is to show that Property~\ref{property:theominimalthing} holds at least in all cases of interest; in order to do  this, we need some further help from valuation theory.

\begin{defin}
  Let $M<N$ be an extension of nontrivially valued fields. A basis $\bla e0,n$ of a  finite-dimensional $M$-vector subspace of $N$ is \emph{separated} iff for every $\bla d0,n\in M$ we have
  \[
    v\Bigl(\sum_{i\le n} d_ie_i\Bigr)=\min_{i\le n} v(d_ie_i)
    \]
  \end{defin}

\begin{fact}[{see~\cite[Lemma~3]{baur} or~\cite[Proposition~12.1(i)]{hhm}}]\label{fact:sepbas}   Let $M<N$ be an extension of nontrivially valued fields. If $M$ is maximally complete, then every finite-dimensional $M$-vector subspace of $N$ has a separated basis.
\end{fact}

The following statement is well-known, but I could not find a reference.
\begin{fact}[Folklore?]\label{fact:folk}
For every $M_0\models \mathsf{RCF}$ there is an $\abs{M_0}^+$-saturated $M\succ M_0$ (in the language of ordered rings) which is maximally complete with respect to the Archimedean valuation and of size at most $\beth_2(\abs{M_0}+2^{\aleph_0})$.
\end{fact}
\begin{proof}[Proof sketch]
 By Fact~\ref{fact:hahn}, it is enough to show that if $\Gamma\models \mathsf{DOAG}$ is $\kappa$-saturated, then so is $\mathbb R((t^\Gamma))$. The cardinality bound then follows by first enlarging $M_0$ so that it contains $\mathbb R$,  then using that $\abs{\mathbb R((t^\Gamma))}\le (2^{\aleph_0})^{\abs \Gamma}=\beth_1(\abs{\Gamma})$, and that there is an $\abs{M_0}^+$-saturated  $\Gamma\succ \Gamma(M_0)$ with $\abs \Gamma\le\beth_1(\abs{M_0})$ by~\cite[Lemma~5.1.4]{changkeisler}.

 Take a partial $1$-type $\Phi(x)=\set{x>a\mid a \in A}\cup \set{x<b\mid b\in B}$, with $\abs{AB}<\kappa$. We may assume that $A$, $B$ only consist of positive elements. Moreover, since $\Gamma$ is $\kappa$-saturated, no  set of size less than $\kappa$ is coinitial in $\mathbb R((t^{\Gamma}))_{>0}$ hence, up to adding a point to $A$,  we may assume that $A\ne \emptyset$. Let $v$ be the Archimedean valuation. If $v(A)>v(B)$, by divisibility and  $\kappa$-saturation of $\Gamma$ there is $\gamma_0\in \Gamma$ with $v(A)>\gamma_0>v(B)$, and then $t^{\gamma_0}\models \Phi(x)$. Otherwise, large enough points of $A$  are all of the form $r_at^{\gamma_0}+\epsilon_a$,  and small enough points of $B$ are all of the form $r_bt^{\gamma_0}+\epsilon_b$,  where $r_-\in \mathbb R\setminus\set0$, $v(\epsilon_-)>\gamma_0$, and $\gamma_0$ is fixed. If there is $r_0$ such that  $\sup_{a\in A} r_a<r_0<\inf_{b\in B} r_b$, then $r_0t^{\gamma_0}\models \Phi(x)$ and we are done. Otherwise, if $\sup_{a\in A} r_a=r_0=\inf_{b\in B} r_b$, take $r_0t^{\gamma_0}$ as our first approximant for a realisation of $\Phi$. Replace $A$ and $B$ with $A_1\coloneqq A-r_0t^{\gamma_0}$ and $B_1=B-r_0t^{\gamma_0}$ and repeat the argument getting a second approximant $r_0t^{\gamma_0}+r_1t^{\gamma_1}$, with $v(\gamma_1)> v(\gamma_0)$. Iterate in the transfinite, where at limit stages we take infinite sums, which is possible since we are in a Hahn field and we are summing over a set with well-ordered support. After fewer than $\kappa$ many steps, we have to realise $\Phi(x)$, because $\abs{AB}<\kappa$.
\end{proof}
\begin{pr}\label{pr:rcfgen}
  In the theory $\mathsf{RCF}$, every invariant type is equidominant to a product of invariant $1$-types.
\end{pr}
\begin{proof}
  Let $p(x)=\tp(c/\monster)$ be $M_0$-invariant and, by enlarging $M_0$, assume that $M_0\succ \mathbb R$.   As usual, we may suppose that $c$ is $\monster$-independent, by working with a basis of $c$ and recovering the rest with a single formula.  Let $b$ be given by point~\ref{point:Emaximal} of  Proposition~\ref{pr:omin1typesgenerate}, satisfying its point~\ref{point:nicerep}. Enlarge $M_0$ further so as to ensure that $\tp(b/\monster)$ is $M_0$-invariant, then use Fact~\ref{fact:folk} to obtain a small $\abs{M_0}^+$-saturated  $M\succ M_0$ which is  maximally complete with respect to the Archimedean valuation $v$. 
  \begin{claim}
Inside the ordered field $M(c)$ generated by $M c$, let $V$ be a finite-dimensional $M$-vector subspace generated by a finite set of monomials  $c^\ell$, for  $\ell\in \omega^{\abs c}$ a multi-index.\footnote{E.g.~if $\abs c=2$ we could have $\ell= (2,7)$ and $c^\ell=c_0^2c_1^7$.}    If $e$ is a separated basis of the $M$-vector space $V$, then it is also a separated basis of the $\monster$-vector space generated by $e$ inside $\monster(c)$, where $M,M(c), \monster, \monster(c)$ are  equipped with the Archimedean valuation.
  \end{claim}
  \begin{claimproof}
    Take a linear combination $\sum_{i\le n} d_ie_i$, with $d_i\in \monster$. Since $e_i\in \dcl(M c)$, we may write $e_i=h_i(c)$, for a suitable $M$-definable function $h_i(x)$. Let $H$ be the (finite) set of parameters outside $M_0$ appearing in the functions $h_i(x)$. Since $M$ is $\abs{M_0}^+$-saturated,  there is $\tilde d\in M$ with  $\tilde d\equiv_{M_0 H} d$. Since $e$ is a separated basis, up to reindexing we have $v(\sum_{i\le n}\tilde d_ih_i(c))=v(\tilde d_nh_n(c))$. Therefore there is a real number $s\in \mathbb R\setminus \set 0$ such that
    \[
     \forall m\in \omega\setminus \set 0\; p(x)\proves \abs*{s-\frac{\sum_{i\le n} \tilde d_ih_i(x)}{\tilde d_nh_n(x)}}<\frac 1m
    \]
    By $M_0$-invariance of $p(x)=\tp(c/\monster)$ we have
        \[
     \forall m\in \omega\setminus \set 0\; p(x)\proves \abs*{s-\frac{\sum_{i\le n}  d_ih_i(x)}{ d_nh_n(x)}}<\frac 1m\qedhere
   \]
 \end{claimproof}
 Apply the rest of Proposition~\ref{pr:omin1typesgenerate}, and work in its notation. So $p(x)=\tp(c/\monster)$, $q(y)=\tp(b/\monster)$,  $r(x,y)=\tp(cb/N_1)$,   $p(x)\cup r(x,y)\proves q(y)$, and 
\[
  q(y)\cup r(x,y)\proves \pi_M(x)=\bigcup_{f\in \mathcal F^{p,1}_M} \tp_{w_f}(f(c)/\monster)\cup \set*{w_f=f(x)\Bigm| f \in \mathcal F^{p,1}_M}
\]
We want to show that $q(y)\cup r(x,y)\proves p(x)$.

By quantifier elimination it is enough to show that $q\cup r$ decides the sign of all polynomials $f(x,d')\in \monster[x]$, where $d'$ is the tuple of coefficients.  Note that, since $c$ is $\monster$-independent, it is $\set{d'}$-independent, hence $p(x)\proves f(x,d')\ne 0$, unless $f(x,d')$ is identically null (in which case there is nothing to do). By Fact~\ref{fact:sepbas}, there is  a separated basis $\bla e0,n$  of the $M$-vector space generated by all the  $c^\ell$ appearing in $f(c,d')$.  We can write $e_i=h_i(c)$, where $h_i(x)$ is an $M$-definable function, and  $c^\ell=\sum_{j\le n} \beta_{j,\ell} e_{j}$, for suitable $\beta_{j,\ell}\in M$. Note that $r(x,y)\proves x^\ell=\sum_{j\le  n}\beta_{j,\ell}h_j(x)$. After replacing, in $f(x,d')$,  each $x^\ell$ with $\sum_{j\le  n}\beta_{j,\ell}h_j(x)$, and collecting the monomials in each $h_j(x)$, we have 
\[
  \models\forall x\; \biggl(\Bigl(\bigwedge_{\ell}x^\ell=\sum_{j\le  n}\beta_{j,\ell}h_j(x)\Bigr)\implica \bigl(f(x,d')=g(d,h(x))\bigr)\biggr)
\]
where $h(x)=(h_0(x),\ldots, h_n(x))$ and $g(d,z)=\sum_{j\le n} d_jz_j$, with $d_j=\sum_\ell d_\ell'\beta_{j,\ell}$. It follows that
\begin{equation}
  r(x,y)\proves f(x,d')=g(d,h(x))\label{eq:lcov}
\end{equation}
Now, since by the Claim $e$ is also a separated basis of the $\monster$-vector space it generates,  $v(f(c,d'))=v(g(d,e))=\min_j v(d_j e_j)$. Suppose, by rearranging $e$, that this equals $v(d_ne_n)$. Define
\[
a\coloneqq1+\sum_{j< n} \frac{d_je_j}{d_ne_n}=  \frac{f(c,d')}{d_n e_n}
\]
Since $v(f(c,d'))=v(e_nd_n)$, we can write $a=s_a+\epsilon_a$, where $s_a\in \mathbb R\setminus \set 0$ and $\epsilon_a$ is \emph{$\mathbb R$-infinitesimal}, i.e.~$\forall m \in \omega\setminus \set 0\;\abs{\epsilon_a}<1/m$. Similarly, for all $j<n$, since $v(d_je_j/d_ne_n)\ge 0$ there are $s_{j}\in \mathbb R$ (now possibly null) and $\epsilon_{j}$ such that for all $m\in \omega\setminus \set 0$  we have  $\abs{\epsilon_{j}}<1/m$ and $(d_je_j)/(d_ne_n)=s_j+\epsilon_j$. Therefore
\[
  \forall m\in \omega\setminus \set 0\;\left( c\models \abs*{\frac{d_j}{d_n}\frac{h_j(x)}{h_n(x)}-s_j}<\frac 1m\right)
\]
This information is, by assumption, in $\pi_M(x)\proves \tp_w((h_j(c)/h_n(c))/\monster)\cup\set{w=h_j(x)/h_n(x)}$, the $h_i(x)$ being $M$-definable. It follows that
\begin{equation}\label{eq:stpart1}
  \forall m\in \omega\setminus \set 0\;\left( q(y)\cup r(x,y)\proves \abs*{\frac{d_j}{d_n}\frac{h_j(x)}{h_n(x)}-s_j}<\frac 1m\right)
\end{equation}
We can write
\begin{equation*}
  s_a+\epsilon_a=  \frac{f(c,d')}{d_ne_n }=1+\sum_{j< n}\frac{d_je_j}{d_ne_n}= 1+\sum_{j<n} (s_j+\epsilon_j)
  =\epsilon'+1+\sum_{j<n}s_{j}
\end{equation*}
where $\epsilon'$ is $\mathbb R$-infinitesimal. Hence $1+\Bigl(\sum_{j<n}s_{j}\Bigr)-s_a$ is $\mathbb R$-infinitesimal and belongs to $\mathbb R$, so it is $0$, yielding   $s_a=1+\sum_{j<n}s_{j}$ and $\epsilon'=\epsilon_a$. Since $\epsilon_a$ is $\mathbb R$-infinitesimal, and $s_a\ne 0$, in particular $\phi(\bla \epsilon0,{n-1})$ holds, where
\begin{equation*}
\phi(\bla u0,{n-1})\coloneqq\abs[\Bigg]{\Bigl(1+\sum_{j< n}(s_j+u_j)\Bigr)-s_a}<\frac {\abs{s_a}}2
\end{equation*}
Note that $\phi(\bla \epsilon0,{n-1})$ holds for \emph{all} $\mathbb 
R$-infinitesimals $\bla \epsilon0,{n-1}$. Hence, if $\Phi(t)\coloneqq\set{\abs t<1/m\mid m\in \omega\setminus\set 0}$, we have $\bigcup_{j<n}\Phi(u_j)\proves \phi(\bla u0,{n-1})$. Therefore, by compactness, for all sufficiently large $m$ we have 
\begin{equation}\label{eq:estimate}
  \models \forall w\; \left[\left(\bigwedge_{j<n}\abs*{w_j-s_j}<\frac 1m\right)\implica \left(\abs[\Bigg]{\Bigl(1+\sum_{j< n}w_j\Bigr)-s_a}<\frac {\abs{s_a}}2\right)\right]
\end{equation}

By~\eqref{eq:lcov}, $r\proves f(x,d')/(d_nh_n(x))=1+\sum_{j<n}{(d_jh_j(x))}/{(d_nh_n(x))}$. This, together with~\eqref{eq:stpart1} and~\eqref{eq:estimate}, yields
\[
q\cup r\proves \abs*{\frac{f(x,d')}{d_nh_n(x)}-s_a}<\frac {\abs{s_a}}2
\]
which in turn implies $q\cup r\proves \abs*{f(x,d')-s_ad_nh_n(x)}< \abs{s_ad_nh_n(x)}/2$, and in particular $q\cup r$ proves that $f(x,d')$ and $s_ad_nh_n(x)$ have same sign. But $\opsc{ed}(\monster)$ decides the sign of both $s_a,d_n\in \monster$ and $\pi_M(x)$ decides the cut, hence the sign, of $h_n(x)$. Therefore, $q\cup r$ decides the sign of  $f(x,d')$.
\end{proof}

By the previous proposition we may apply Theorem~\ref{thm:omincharmodgen} to $\mathsf{RCF}$, and obtain the following result by Proposition~\ref{pr:convexringsswort} and Corollary~\ref{co:subrings}.

\begin{thm}\label{thm:rcf}
In $\mathsf{RCF}$, the domination monoid $\invtilde$ is well-defined, coincides with $\invbar$, and $(\invtilde, \otimes, \doms, \wort)\cong (\pfin(X), \cup, \supseteq, D)$, where $X$ is the set of invariant convex subrings of $\monster$.
\end{thm}

\begin{rem}\label{rem:rank_follow_up}
The Archimedean valuation induces a bijection between invariant convex subrings of $\monster$ and invariant convex subgroups of its value group $\Gamma(\monster)$, a monster model of $\mathsf{DOAG}$, hence an isomorphism between the respective domination monoids. As we saw in Remark~\ref{rem:rank}, these subgroups correspond to invariant cuts in the rank of $\Gamma(\monster)$ under the (ordered group-theoretic) Archimedean valuation $\mathfrak v$.
In conclusion, we have isomorphisms
\[
  \invtilde\cong\invtildeof{\Gamma(\monster)}\cong \invtildeof{\mathfrak v(\Gamma(\monster))}
\]
and the domination monoid of a monster model $\monster$ of $\mathsf{RCF}$ is the upper semilattice of finite subsets of the set of invariant cuts in the Archimedean rank of the Archimedean value group of $\monster$.
\end{rem}

\section{Real closed valued fields}\label{sec:arcvf}
\begin{assdrp}
  From now on, we drop the blanket assumption that $T$ is o-minimal.
\end{assdrp}

The reason $\invbar$ was first introduced in the unstable context was to prove that, in (each completion of) the theory $\mathsf{ACVF}$ of algebraically closed valued fields, it decomposes as a direct product of the equidominance monoids of the residue field and value group. The key ingredient to this result was shown in~\cite{ehm} to also hold in the weakly o-minimal theory $\mathsf{RCVF}$ of real closed valued fields with proper convex valuation ring. Since the residue field of any model of $\mathsf{RCVF}$ is equipped with the structure of a pure real closed field, the work carried out in Section~\ref{sec:rcf} allows us to complete the computation of the domination monoid in $\mathsf{RCVF}$, provided we show that $\otimes$ respects $\equidom$. In this section we spell out the details, and show that $\equidom$ and $\domeq$ coincide in $\mathsf{RCVF}$. Essentially the same arguments, \emph{mutatis mutandis}, show that $\domeq$ coincides with $\equidom$ and respects $\otimes$ in $\mathsf{ACVF}$ as well; see~\cite[Theorem~5.2.22]{mythesis}.

There is a variety of languages in which $\mathsf{RCVF}$ can be formulated. One common choice is to work with three sorts $K,k,\Gamma$, where $K$ is the actual valued field, $k$ is the residue field, and $\Gamma$ the value group.   The first two sorts are ordered fields, while $\Gamma$ comes with the structure of an ordered abelian group.  Strictly speaking, $\Gamma$ also includes a constant symbol  for the valuation of $0$, which is not part of the ordered group structure. It is customary to abuse the notation, talking of $\Gamma$ as if it were an ordered group. The sorts are connected by the valuation $v\from K\to \Gamma$ and the modified residue map $\operatorname{Res}\from K^2\to k$ sending $(x,y)$ to the residue class of $x/y$ if the latter is in the valuation ring, and to $0$ otherwise. 

It is well-known that such a language is not enough to eliminate imaginaries, and extra sorts are needed for this purpose.   In what follows, we will work in a language with elimination of quantifiers and imaginaries including the sorts $K$, $k$, and $\Gamma$. The particular language is not important: by~\cite[Remark~2.3.33]{mythesis}, we may as well work in $T^\eq$. We refer the reader to~\cite{mellor} for a description of an appropriate language where elimination of imaginaries holds.

In the case of $\mathsf{RCVF}$, an important part in the analysis of the domination monoid is played by the \emph{full embeddedness} of the sorts $k$ and $\Gamma$: the subsets of every cartesian power of $k(\monster)$ which are definable (with parameters) in $\monster$ are already definable in $k(\monster)$ equipped with the ordered field structure, and the analogous statement holds for $\Gamma$.  It is easy to show that the domination monoid of a fully embedded sort embeds in that of the full structure; the reader may find a proof of the following fact  in~\cite[Proposition~2.3.31]{mythesis}.  
\begin{fact}\label{fact:stabemb}
Suppose that $Y(\monster)$ is  a fully embedded sort of $\monster$. The natural map $\iota\from S(Y(\monster))\to S(\monster)$, sending a type over $Y(\monster)$ to the unique type over $\monster$ it entails, sends invariant types to invariant types and, when restricted to such types, is an injective $\otimes$-homomorphism. Moreover, $\iota$ induces an embedding of posets $\invtildeof{Y(\monster)} \into \invtilde$ and an embedding of sets  $\invbarof{Y(\monster)} \into \invbar$. Each of these embeddings is a $\wort$-homomorphism, a $\nwort$-homomorphism,  and,  if $\otimes$ respects $\doms$ [resp.~$\equidom$],  an embedding of monoids.
\end{fact}

\begin{fact}[{see~\cite{vddvf} and~\cite[Lemma~3.13]{mellor}}]\label{fact:xcvfstart}
  In $\mathsf{RCVF}$ the following hold.
  \begin{enumerate}
  \item We have $\Gamma\models \mathsf{DOAG}$ and $k\models \mathsf{RCF}$.
  \item\label{point:xcvfwort} If $p\in S_{k^n}(\monster)$ and  $q\in S_{\Gamma^m}(\monster)$, then $p\wort q$.
  \item    The structures $k(\monster)$ and $\Gamma(\monster)$ are fully embedded.
  \end{enumerate}
\end{fact}

\begin{notation}
We denote by $k(B)$ [resp.~$\Gamma(B)$] the set of points of $\dcl(B)$ which belong to the sort $k$ [resp.~$\Gamma$].
\end{notation}

 \begin{fact}\label{fact:dclsorts}
Suppose that all coordinates of $p,q\in \invtypes(\monster,M)$ are in the valued field sort $K$ and $M$ is maximally complete. If  $(a,b)\models p\otimes q$,  then
  \begin{equation*}
    k(M ab)=k\bigl(k(M a), k(M b)\bigr)\qquad 
    \Gamma(M ab)=\Gamma\bigl(\Gamma(M a), \Gamma(M b)\bigr)
  \end{equation*}
\end{fact}
\begin{proof}[Proof sketch]
  For $X\subseteq k$ and $Y\subseteq \Gamma$,  denote by $F(X)$ the field generated by $X$, and by $G(Y)$ the group generated by $Y$.   By invariance, $k(M a)$ and $k(M b)$ are linearly disjoint over $k(M)$, and $\Gamma(M a)\cap\Gamma(M b)=\Gamma(M)$. Therefore, we may apply~\cite[Proposition~12.11(ii)]{hhm} to $\dcl(M a)$ and $\dcl(M b)$, and obtain
    \begin{equation*}
    k(M ab)=F\bigl(k(M a), k(M b)\bigr)\qquad 
    \Gamma(M ab)=G\bigl(\Gamma(M a), \Gamma(M b)\bigr)
  \end{equation*}
  To conclude, observe that 
  \begin{gather*}
    k\bigl(k(M a), k(M b)\bigr)\subseteq  k(M ab)=  F\bigl(k(M a), k(M b)\bigr)\subseteq k\bigl(k(M a), k(M b)\bigr)\phantom{\qedsymbol}
    \\
    \Gamma\bigl(\Gamma(M a), \Gamma(M b)\bigr)\subseteq  \Gamma(M ab)=  G\bigl(\Gamma(M a), \Gamma(M b)\bigr)\subseteq \Gamma\bigl(\Gamma(M a), \Gamma(M b)\bigr)\qedhere
\end{gather*}
\end{proof}

We now proceed to compute $\invtilde$ in $\mathsf{RCVF}$.  The statement and proof of Theorem~\ref{thm:rcvf} below are essentially~\cite[Corollary~12.14]{hhm}, except the latter  worked in $\mathsf{ACVF}$, considered $\invbar$ only,  took its well-definedness and commutativity for granted, and used \cite[Corollary~12.12]{hhm}, which in the case of $\mathsf{RCVF}$ is replaced by the fact below.

\begin{fact}[\!\!{\cite[Corollary~2.8]{ehm}}]\label{fact:rcvfgen}
 Let $M,B, \monster$ be contained in a monster model  of $\mathsf{RCVF}$.  Let $M$ be maximally complete, $M\subseteq B=\dcl(B)$, and $M\subseteq \monster$, with $k(B), k(\monster)$ linearly disjoint over $k(M)$, and $\Gamma(B)\cap \Gamma(\monster)=\Gamma(M)$. Then $\tp(\monster/M, k(B),\Gamma(B))\proves \tp(\monster/B)$. 
 \end{fact}
Recall that we know how to characterise $\invtildeof{k(\monster)}$ and $\invtildeof{\Gamma(\monster)}$: see   Theorem~\ref{thm:rcf}  and Theorem~\ref{thm:doag} respectively.
\begin{thm}\label{thm:rcvf}
  In $\mathsf{RCVF}$ the domination monoid is well-defined, and we have \[
    \invbar=    \invtilde\cong\invtildeof{k(\monster)}\oplus \invtildeof{\Gamma(\monster)}
  \]
\end{thm}
\begin{proof}
  Let $p(x)\in \invtypes(\monster, M)$; by~\cite[Proposition~3.6 and Corollary~4.14]{vddvf}, up to enlarging $M$ not beyond size $\beth_1(\abs M)$, we may assume it is maximally complete.  In some $\monster_1\satext \monster$, let $b\models p$ and $B=\dcl(M b)$. By $M$-invariance of $p$, the fields $k(B), k(\monster)$ are linearly disjoint over $k(M)$, and  $\Gamma(B)\cap \Gamma(\monster)=\Gamma(M)$. Apply Fact~\ref{fact:rcvfgen},  recalling that its conclusion is  to be understood modulo the elementary diagram $\opsc{ed}(\monster_1)$.  By making explicit which parts of it we are using, we find that,  working just modulo $T$,
  \[
    \tp(\monster, M, k(B),\Gamma(B)/\emptyset)\cup \tp(k(B), \Gamma(B), B,M/\emptyset)\proves \tp(\monster,B/\emptyset)
  \]
  When we work modulo the elementary diagram $\opsc{ed}(\monster)$, this becomes
  \begin{equation}
    \tp(k(B),\Gamma(B)/\monster)\cup\tp(k(B),\Gamma(B),B/M)\proves \tp(B/\monster)\label{eq:decreph}
  \end{equation}
Recall that $B=\dcl(M b)$;  since $b$ is finite, by (the proofs of)~\cite[Corollary~11.9 and Corollary~11.16]{hhm}  there is a finite tuple $\tilde b$ from $K$ such that  $K(M b)=K(M \tilde b)$. By~\cite[Proposition~8.1(i)]{mellor},  inside $K$, $\dcl$ coincides with $\dcl$ in the sense of the restriction to the ordered field language.   It follows that there are finite tuples $b_k$ and $b_\Gamma$, in a cartesian power of $k(B)$ and $\Gamma(B)$ respectively, such that $k(B)=k(k(M) b_k)$ and $\Gamma(B)=\Gamma(\Gamma(M)b_\Gamma)$. Note that there are $M$-definable functions sending $b$ to  $b_k$ and $b_\Gamma$. Let $p_k\coloneqq\tp(b_k/\monster)$ and $p_\Gamma\coloneqq\tp(b_\Gamma/\monster)$; since $p_k\wort p_\Gamma$, we have $p_k\otimes p_\Gamma=p_k\cup p_\Gamma=p_\Gamma\otimes p_k$. Define  $r\coloneqq\tp(b_k,b_\Gamma,b/M)$.  By~\eqref{eq:decreph} we have  $p_k\cup p_\Gamma\cup r\proves p$.   Since $p\cup r\proves p_k\cup p_\Gamma$, as $r$ says the latter is a pushforward of $p$, we obtain $p\equidom p_k\otimes p_\Gamma$. 
   \begin{claim}
   For all $p,q\in\invtypes(\monster)$, we have $p\otimes q\equidom q\otimes p$.   
 \end{claim}
 \begin{claimproof}
Let $p,q\in \invtypes(\monster, M)$. Again by~\cite[Corollary~4.14]{vddvf}, we may assume that $M$ is maximally complete.    Assume first that all coordinates of both $p$ and $q$  are in $K$, and choose suitable $p_k,q_k,p_\Gamma, q_\Gamma$ as above; these are not unique, but since we need to use them in Fact~\ref{fact:rcvfgen}, we only care about their realisations up to definable closure. The same applies to e.g.~$(p\otimes q)_k$, and  Fact~\ref{fact:dclsorts}, 
   ensures that we may take $(p\otimes q)_k\coloneqq p_k \otimes q_k$ and $(p\otimes q)_\Gamma\coloneqq p_\Gamma \otimes q_\Gamma$, and  have $p\otimes q\equidom p_k\otimes q_k\otimes p_\Gamma\otimes q_\Gamma$. By applying the same argument to $q\otimes p$, \emph{and using the same $p_k, q_k, p_\Gamma, q_\Gamma$}, we obtain 
\begin{equation}\label{eq:fermino}
p\otimes q\equidom p_k\otimes q_k\otimes p_\Gamma\otimes q_\Gamma \qquad  q\otimes p\equidom q_k\otimes p_k\otimes q_\Gamma\otimes p_\Gamma
\end{equation}
Before concluding that these two types are equidominant, we show that a similar situation may be arranged in the case where some coordinates of $p,q$ are in an imaginary sort. By~\cite[Theorem~4.5]{ehm}, if $a$ is an imaginary tuple then there is a tuple $\tilde a$ from the sort $K$ such that $a\in\dcl(M \tilde a)$ and
\begin{equation}
k(M a)=k(M \tilde a) \qquad \Gamma(M a)=\Gamma(M \tilde a)\label{eq:resolve}
\end{equation}
Let $a\models p$, let  $f$ be an $M$-definable function such that $f(\tilde a)=a$,  denote $\tilde p\coloneqq\tp(\tilde a/\monster)$, and observe that $f_*\tilde p=p$.  Given $b\models q$, define $\tilde q$ and $g$ analogously.

Let $(a', b')\models \tilde p\otimes \tilde q$, and let $a\coloneqq f(a')$ and $b\coloneqq g(b')$. Since $a'\models \tilde p\invext \monster b'$ and $a=f(a')$,  by~\cite[Lemma~1.13]{invbartheory} we have  $a\models p\invext \monster b'$. In particular $a\models p\invext \monster b$, and therefore  $(a,b)\models p\otimes q$. Since the fact that~\eqref{eq:resolve} holds is a property of $\tilde p$, and similarly for $\tilde q$, we may take  $\tilde a\coloneqq a'$ and $\tilde b\coloneqq b'$. By Fact~\ref{fact:dclsorts} and~\eqref{eq:resolve},
\[
  k(k(M a)k(M b))\subseteq k(M ab)\subseteq k(M\tilde a \tilde b)=k(k(M \tilde a)k(M \tilde b))=k(k(M a)k(M b))
\]
and similarly for $\Gamma$.  Therefore we may take $\tilde a\tilde b$ as  $\widetilde{ab}$, that is,
\begin{equation}
    k(M ab)=k(M \tilde a\tilde b)   \qquad
    \Gamma(M ab)=\Gamma(M \tilde a\tilde b)\label{eq:resolved}
\end{equation}
Let $B\coloneqq \dcl(M \tilde{a}\tilde{b})$, and let $\tilde p_k,\tilde q_k, \tilde p_\Gamma, \tilde q_\Gamma$ be defined as above. By~\eqref{eq:fermino}, $\tilde p\otimes \tilde q$ is equidominant to $\tilde p_k\otimes\tilde q_k\otimes \tilde p_\Gamma\otimes \tilde q_\Gamma$ and, using again~\eqref{eq:decreph}, so is $p\otimes q$, because by~\eqref{eq:resolved} we may take $\tilde p_k\otimes\tilde q_k$ as $(p\otimes q)_k$, and similarly for $\Gamma$.  Use \emph{the same} four types, and obtain similarly that $\tilde q\otimes \tilde p$ and $q\otimes p$ are equidominant to $\tilde q_k\otimes\tilde p_k\otimes \tilde q_\Gamma\otimes \tilde p_\Gamma$. Therefore, if we set $p_k\coloneqq \tilde p_k$, and similarly for the other three types, then~\eqref{eq:fermino} holds for imaginary tuples as well. 

By Theorem~\ref{thm:rcf} and Theorem~\ref{thm:doag} the product $\otimes$  is commutative modulo $\equidom$ in $\mathsf{RCF}$ and $\mathsf{DOAG}$.  Using this, Fact~\ref{fact:stabemb},  and Fact~\ref{fact:xcvfstart}, we obtain
  \begin{multline*}
    p\otimes q\equidom  p_k\otimes q_k\otimes p_\Gamma\otimes q_\Gamma\equidom q_k\otimes p_k\otimes q_\Gamma\otimes p_\Gamma\equidom q\otimes p\qedhere
  \end{multline*}
\end{claimproof}
By the Claim and Corollary~\ref{co:commthenwd},  $\otimes$ respects both $\doms$ and $\equidom$.  Because $k(\monster)$ and $\Gamma(\monster)$ are fully embedded in $\monster$, by Fact~\ref{fact:stabemb} we have embeddings $\invbarof{k(\monster)}\into \invbar$ and $\invbarof{\Gamma(\monster)}\into \invbar$.  By point~\ref{point:xcvfwort} of Fact~\ref{fact:xcvfstart} and Fact~\ref{fact:wortpreserved} we  have an embedding $\invbarof{k(\monster)}\oplus \invbarof{\Gamma(\monster)}\into \invbar$. By the decomposition  $p\equidom p_k\otimes p_\Gamma$, this embedding is surjective.

The only statement left to prove is that  $\domeq$ equals $\equidom$. Recall that this is the case in $\mathsf{RCF}$ and $\mathsf{DOAG}$. Suppose that $p_0\nequidom p_1$. By the  isomorphism $\invbar\cong\invbarof{k(\monster)}\oplus \invbarof{\Gamma(\monster)}$, together with  Theorem~\ref{thm:rcf}  and Theorem~\ref{thm:doag}, there are $i<2$ and a $1$-type $q$ in either  $\invtypes_\Gamma(\monster)$ or $\invtypes_k(\monster)$ such that $p_i\doms q$ but $p_{1-i}\ndoms q$. Therefore, $p_0\ndomeq p_1$, and we are done.
\end{proof}

\section{Open questions}\label{sec:questions}
\begin{question}
  Let $T$ be arbitrary, and suppose that $p_0\wort q$ and $p_1\wort q$. Is it true that $p_0\otimes p_1 \wort q$? What if we also assume $\mathsf{NIP}$?
\end{question}
 The answer is known to be positive under distality (Lemma~\ref{lemma:distalperpotimes}), as well as under stability, because for global types in stable theories weak orthogonality coincides with orthogonality.  The next question asks whether the Idempotency Lemma can be improved.
\begin{question}
In an o-minimal $T$, let  $p(x)\in \invtypes_1(\monster, M)$  have small cofinality on the right. If $b^0\models p$, is $p(\dcl(M b^0))$ cofinal in $p(\dcl(\monster b^0))$?
\end{question}

A large part  of the material in Section~\ref{sec:reduction} works under the sole assumption of o-minimality, and I do not know of any o-minimal theory not satisfying Property~\ref{property:1tpgen}. I have not seen a proof that this always holds either and, while this is not the only possible approach to the problem, the reduction to Property~\ref{property:theominimalthing} is a step towards such a proof.
\begin{question}
Does  Property~\ref{property:theominimalthing} hold in every o-minimal theory, or at least in o-minimal expansions of $\mathsf{DOAG}$? More generally, in the setting of Proposition~\ref{pr:omin1typesgenerate}, does $\pi_M(x)\proves p(x)$?
\end{question}
We may also ask if some control over bases of invariance is possible.
  \begin{question}
    Let $T$ be o-minimal, and let  $p$ be a global $M$-invariant type which is domination-equivalent to a product of invariant $1$-types. Is $p$ necessarily domination-equivalent to a product of $M$-invariant $1$-types?
  \end{question}

We saw that in o-minimal  groups and  fields with no extra structure, generators of the domination monoid correspond to invariant convex subgroups and subrings. Of course the particular description of a set of generators will depend on the particular theory at hand, so we state the following problem for a particular structure; we phrase it in a way that makes sense even if Property~\ref{property:1tpgen} turns out to fail.
\begin{prob}
Identify a nice maximal set of pairwise weakly orthogonal invariant $1$-types in monster models of the theory of $\mathbb R_\mathrm{exp}$.
\end{prob}

\end{document}